\documentclass[12pt]{amsart}
\usepackage{amsmath}
\usepackage{amsfonts}
\usepackage{amssymb}
\usepackage{amscd}
\usepackage{mathtools}
\usepackage{amsxtra}
\usepackage{amsthm}
\usepackage{graphicx}
\usepackage[abbrev,alphabetic]{amsrefs}
\usepackage[dvipsnames]{xcolor}
\usepackage{soul}
\setstcolor{red}

\usepackage{booktabs}
\usepackage{multirow}
\newtagform{tiny}{\tiny(}{)}
\usepackage{threeparttable}
\usepackage{aliascnt}

\usepackage{mathtools}
\usepackage{hyperref}
\usepackage[margin=1.25in]{geometry}
\usepackage{mathrsfs}

\usepackage{amsthm}
\usepackage{comment}
\usepackage[all,cmtip]{xy}
\usepackage{tikz-cd}
\usetikzlibrary{cd}

\tikzcdset{
  cells={font=\everymath\expandafter{\the\everymath\displaystyle}},
}

\usepackage[all]{xy}

\usepackage{cleveref}

\makeatletter
\def\@tocline#1#2#3#4#5#6#7{\relax
  \ifnum #1>\c@tocdepth 
  \else
    \par \addpenalty\@secpenalty\addvspace{#2}%
    \begingroup \hyphenpenalty\@M
    \@ifempty{#4}{%
      \@tempdima\csname r@tocindent\number#1\endcsname\relax
    }{%
      \@tempdima#4\relax
    }%
    \parindent\z@ \leftskip#3\relax \advance\leftskip\@tempdima\relax
    \rightskip\@pnumwidth plus4em \parfillskip-\@pnumwidth
    #5\leavevmode\hskip-\@tempdima
      \ifcase #1
       \or\or \hskip 1em \or \hskip 2em \else \hskip 3em \fi%
      #6\nobreak\relax
    \hfill\hbox to\@pnumwidth{\@tocpagenum{#7}}\par
    \nobreak
    \endgroup
  \fi}
\makeatother

\newcommand{\rup}[1]{\lceil #1 \rceil}

\renewcommand{\mod}{\ \textrm{mod}\ }

\newcommand{\Z}{\mathbb{Z}}
\newcommand{\Q}{\mathbb{Q}}

\newcommand{\m}{\mathfrak{m}}

\newcommand{\bn}{\boldsymbol{n}}
\newcommand{\bh}{\boldsymbol{h}}
\newcommand{\bolde}{\boldsymbol{e}}

\newcommand{\sQ}[2]{Q_{#1,#2}}

\def\var{\overline}

\DeclareMathOperator{\Hom}{Hom}

\DeclareMathOperator{\Ker}{Ker}

\DeclareMathOperator{\sht}{ht}

\DeclareMathOperator{\fpt}{fpt}

\DeclareMathOperator{\ppt}{ppt}

\renewcommand{\Im}{\mathrm{Im}}

\theoremstyle{plain}
\newtheorem{theorem}{Theorem}[section]

\newtheorem{theoremA}{Theorem}

\crefname{theoremA}{Theorem}{Theorems}
\Crefname{theoremA}{Theorem}{Theorems}

\newaliascnt{lemma}{theorem}
\newtheorem{lemma}[lemma]{Lemma}
\aliascntresetthe{lemma}
\crefname{lemma}{Lemma}{Lemmas}
\Crefname{lemma}{Lemma}{Lemmas}

\newaliascnt{proposition}{theorem}
\newtheorem{proposition}[proposition]{Proposition}
\aliascntresetthe{proposition}
\crefname{proposition}{Proposition}{Propositions}
\Crefname{proposition}{Proposition}{Propositions}

\newaliascnt{prop}{theorem}
\newtheorem{prop}[prop]{Proposition}
\aliascntresetthe{prop}
\crefname{prop}{Proposition}{Propositions}
\Crefname{prop}{Proposition}{Propositions}

\newaliascnt{corollary}{theorem}

\aliascntresetthe{corollary}
\crefname{corollary}{Corollary}{Corollaries}
\Crefname{corollary}{Corollary}{Corollaries}

\newaliascnt{cor}{theorem}

\aliascntresetthe{cor}
\crefname{cor}{Corollary}{Corollaries}
\Crefname{cor}{Corollary}{Corollaries}

\newaliascnt{claim}{theorem}
\newtheorem{claim}[claim]{Claim}
\aliascntresetthe{claim}
\crefname{claim}{Claim}{Claims}
\Crefname{claim}{Claim}{Claims}

 \newtheorem*{claim*}{Claim}

\theoremstyle{definition}
\newaliascnt{definition}{theorem}
\newtheorem{definition}[definition]{Definition}
\aliascntresetthe{definition}
\crefname{definition}{Definition}{Definitions}
\Crefname{definition}{Definition}{Definitions}

\newaliascnt{dfn}{theorem}

\aliascntresetthe{dfn}
\crefname{dfn}{Definition}{Definitions}
\Crefname{dfn}{Definition}{Definitions}

\newaliascnt{notation}{theorem}
\newtheorem{notation}[notation]{Notation}
\aliascntresetthe{notation}
\crefname{notation}{Notation}{Notations}
\Crefname{notation}{Notation}{Notations}

\newaliascnt{example}{theorem}
\newtheorem{example}[example]{Example}
\aliascntresetthe{example}
\crefname{example}{Example}{Examples}
\Crefname{example}{Example}{Examples}

\theoremstyle{remark}
\newaliascnt{remark}{theorem}
\newtheorem{remark}[remark]{Remark}
\aliascntresetthe{remark}
\crefname{remark}{Remark}{Remarks}
\Crefname{remark}{Remark}{Remarks}

\newaliascnt{setting}{theorem}

\aliascntresetthe{setting}
\crefname{setting}{Setting}{Settings}
\Crefname{setting}{Setting}{Settings}

\newtheorem*{ackn}{Acknowledgements}

\newenvironment{claimproof}[0]
  {%
   \paragraph{\it Proof.}%
  }
  {%
    \hfill$\blacksquare$%
  }

\numberwithin{equation}{section}

\crefname{theorem}{Theorem}{Theorems}
\crefname{proposition}{Proposition}{Propositions}
\crefname{lemma}{Lemma}{Lemmas}
\crefname{corollary}{Corollary}{Corollaries}
\crefname{conjecture}{Conjecture}{Conjectures}
\crefname{claim}{Claim}{Claims}
\crefname{notation}{Notation}{Notations}
\crefname{remark}{Remark}{Remarks}
\crefname{example}{Example}{Examples}
\crefname{definition}{Definition}{Definitions}
\crefname{theoremA}{Theorem}{Theorems}

\title{Perfectoid pure thresholds of lifts of rational double points}

\author{Teppei Takamatsu}
\address{Department of Mathematics, Faculty of Science,
Saitama University,
255 Shimo-Okubo, Sakura-ku,
Saitama-shi, Saitama 338-8570,
Japan}
\email{teppeitakamatsu.math@gmail.com}

\author{Shou Yoshikawa}
\address{Institute of Science Tokyo, Tokyo 152-8551, Japan}
\email{yoshikawa.s.9fe9@m.isct.ac.jp}

\begin{document}

\begin{abstract}
We study the perfectoid pure threshold with respect to $p$, an invariant of singularities in mixed characteristic $(0,p)$ arising from perfectoid purity.
In this paper, we compute perfectoid pure thresholds for lifts of rational double points.
We show that the set of such thresholds is contained in $\Q$ and satisfies the ascending chain condition.
In characteristic $2$, all reciprocals of positive integers occur, and $0$ is the unique accumulation point.
\end{abstract}

\maketitle

\section{Introduction}

The theory of singularities in mixed characteristic has recently undergone remarkable development, as demonstrated in works such as \cite{MSTWW}, \cite{BMPSTWW24}, and \cite{HLS}. 
This line of research has led to profound applications in both algebraic geometry and commutative algebra, including the mixed characteristic minimal model program \cite{BMPSTWW23}, \cite{TY23}, as well as Skoda-type theorems \cite{HLS}.

In this paper, we study an invariant of rings in characteristic $(0,p)$ called the \emph{perfectoid pure threshold}. 
The notion of perfectoid purity was introduced in \cite{p-pure} in the study of singularities in mixed characteristic, and the perfectoid pure threshold can be viewed as an analogue of the $F$-pure threshold in positive characteristic, introduced in \cite{TW04}, and of the log canonical threshold in characteristic zero.
Note that closely related mixed characteristic thresholds have also been investigated recently (see \cite{Rod}, \cite{cai}, and \cite{benozzo}).

The second author introduced methods for computing perfectoid pure thresholds with respect to $p$ in \cite{Yoshikawa25} and \cite{Yoshikawa25-cri}. 
Using these methods, the second author proved the following result.

\begin{theorem}[\cite{Yoshikawa25-cri}*{Theorem~C}]\label{intro:rationality}
Let $d$ be a non-negative integer.
Then the following two sets coincide:
\begin{itemize}
\item the set $\mathcal{P}_{d,p}$ of $\ppt(R,p)$ such that $(R,\m)$ is a regular local ring of dimension $d$ with $p \in \m$ and $R/pR$ is $F$-finite, and
\item the set $\mathcal{T}_{d,p}$ of $\fpt(R,f)$ such that $(R,\m)$ is a regular $F$-finite local ring of dimension $d$ of characteristic $p$ and $0 \neq f \in \m$.
\end{itemize}
\end{theorem}

\noindent
Since the set $\mathcal{T}_{d,p}$ is known to satisfy several discreteness properties—namely rationality, the ascending chain condition, and restrictions on accumulation points \cite{BMS}, \cite{Sato21}, \cite{Sato23}—the same properties also hold for $\mathcal{P}_{d,p}$.

In this paper, we compute perfectoid pure thresholds  with respect to $p$ for not necessarily regular lifts of rational double points (RDPs) and prove the following result, which generalizes such discreteness phenomena.

\begin{theoremA}[\Cref{ACC}]\label{intro:ACC}
Let $p$ be a prime number and let $k$ be an algebraically closed field of characteristic $p$.
Consider the set
\[
\Sigma := \{ \ppt(R,p) \mid \text{$R$ is a $W(k)$-lift of an RDP of characteristic $p$ over $k$} \}.
\]
Then $\Sigma$ is contained in $\Q$ and satisfies the ascending chain condition.
Furthermore, if $p=2$, the set $\Sigma$ contains all numbers of the form $1/m$ with $m \in \Z_{\ge 1}$.
Moreover, every accumulation point of $\Sigma$ is equal to $0$.
\end{theoremA}

\noindent
Furthermore, we give \Cref{table:ppts}, which lists all perfectoid pure thresholds for lifts of RDPs other than those of type $D$.

To prove Theorem~\ref{intro:ACC}, we introduce the notion of \emph{multi-height} together with Fedder-type methods for computing it (see Section~\ref{section:multi-height}).
This notion is useful for computing $\ppt(-,p)$ for lifts of quasi-$F$-split complete intersection local rings.

\begin{ackn}
The first author was supported by JSPS KAKENHI Grant Number JP25K17228.
The second author was supported by JSPS KAKENHI Grant Number JP24K16889.
Some of the computations in this paper were carried out using Macaulay2, and the authors are grateful to its developers.
\end{ackn}

\section{Multi-height and perfectoid pure threshold of \texorpdfstring{$p$}{p}}
\label{section:multi-height}

\subsection{Definition of multi-height and Fedder-type criterion}

In this subsection, we introduce the \emph{multi-height}, an invariant for computing $\mathrm{ppt}(-,p)$ for quasi-$F$-split $\mathbf{Z}_{(p)}$-algebras. 
It should be viewed as a refined version of the mixed-characteristic quasi-$F$-split height (\cite{Yoshikawa25}*{Definition 4.4}). Moreover, as an extension of \cite{kty}*{Theorem 4.11} and \cite[Theorem~4.13]{Yoshikawa25}, we prove a criterion \Cref{multi-fedder} (a Fedder-type criterion) for computing the multi-height.

\begin{definition}
Let $R$ be a $\Z_{(p)}$-algebra.
For $r \in \Z_{\geq 0}$ and $\boldsymbol{n}:=(n_0,n_1,\ldots,n_r) \in \Z_{\geq 1}^{r+1}$, we define the $\var{R}$-module $\sQ{R}{\bn}$ along with homomorphisms
\[
\Psi_{R,\bn}\colon \sQ{R}{(n_0,\ldots,n_{r-1})} \to \sQ{R}{\bn},\quad V^{n_0-1} \colon F^{n_0}_*\sQ{R}{(n_1,\ldots,n_r)} \to \sQ{R}{\bn}
\]
Here, when $r=0$, we regard the domain of $\Psi_{R,\bn}$ as $\overline{R}$.
We define them inductively as follows:
\begin{itemize}
    \item For $r=0$, we set $\sQ{R}{(n_0)} := Q_{R,n_0}$ and $\Psi_{R,(n_0)} := \Phi_{R,n_0}$.
    These are defined in \cite{Yoshikawa25}*{Definition~4.1}.
    \item For $r=1$, we define $\sQ{R}{(n_0,n_1)}$, $\Psi_{R,(n_0,n_1)}$, and $V^{n_0-1}$ via the following pushout diagram in the category of $\var{R}$-modules:
\begin{equation}\label{eq:push1}
    \begin{tikzcd}[column sep=1.5cm]
    F^{n_0}_*\var{R} \arrow[r,"F^{n_0}_*\Psi_{R,(n_1)}"] \arrow[d,"V^{n_0-1}"] & F^{n_0}_*\sQ{R}{(n_1)} \arrow[d,"V^{n_0-1}"] \\
    \sQ{R}{(n_0)} \arrow[r,"\Psi_{R,(n_0,n_1)}"] & \sQ{R}{(n_0,n_1)}, \arrow[ul, phantom, "\ulcorner", very near start]
\end{tikzcd}
\end{equation}
    \item For $r \geq 2$, we define $\sQ{R}{\bn}$, $\Psi_{R,\bn}$, and $V^{n_0-1}$ via the following pushout diagram in the category of $\var{R}$-modules:
\begin{equation}\label{eq:push2}
\begin{tikzcd}[column sep=2.5cm]
    F^{n_0}_*\sQ{R}{(n_1,\ldots,n_{r-1})} \arrow[r,"F^{n_0}_*\Psi_{R, (n_1,\ldots,n_r)}"] \arrow[d,"V^{n_0-1}"] & F^{n_0}_*\sQ{R}{(n_1,\ldots,n_r)} \arrow[d,"V^{n_0-1}"] \\
    \sQ{R}{(n_0,\ldots,n_{r-1})} \arrow[r,"\Psi_{R,\bn}"] & \sQ{R}{\bn}. \arrow[ul, phantom, "\ulcorner", very near start]
\end{tikzcd}
\end{equation}
\end{itemize}

Furthermore, we define the map $\Phi_{R,\bn} \colon \var{R} \to \sQ{R}{\bn}$ as the composition
\[
\Phi_{R,\bn} \colon \var{R} \xrightarrow{\Psi_{R,(n_0)}} \sQ{R}{(n_0)} \xrightarrow{\Psi_{R,(n_0,n_1)}} \sQ{R}{(n_0,n_1)} \xrightarrow{\Psi_{R,(n_0,n_1,n_2)}} \cdots \xrightarrow{\Psi_{R,\bn}} \sQ{R}{\bn}.
\]
\end{definition}

\begin{definition}
Let $R$ be a Noetherian $\Z_{(p)}$-algebra such that $\var{R}$ is $F$-finite.
For $\bh := (h_0,h_1,\ldots) \in \prod \Z_{\geq 1}$, we say that $R$ has the \emph{quasi-$F$-split multi-height $\bh$} (or simply, multi-height $\bh$) if for every $r \geq 0$, 
\[
h_r = \inf \{h \geq 1 \mid \Phi_{R,(h_0,\ldots,h_{r-1},h)}\ \text{splits} \}.
\]
\end{definition}

\begin{remark}\label{remark:ineq-h}
If $R$ is quasi-$F$-split with $\sht(R)=h$, then we have $h_0=h$ and $h_i \le h$ for all $i \ge 1$.
Indeed, the equality $h_0=h$ follows directly from the definition.
Since the morphism $\Psi_{R, (h)} \colon \var{R} \to \sQ{R}{(h)}$ splits, the map 
$\Psi_{R,(h_{i-1},h)} \colon \sQ{R}{(h_{i-1})} \to 
\sQ{R}{(h_{i-1},h)}$ also splits
by the pushout diagram \eqref{eq:push1}.
Using \eqref{eq:push2} and repeating the above argument, we see that the homomorphism
\[
\Psi_{R,(h_0,h_1,\ldots,h_{i-1},h)} \colon \sQ{R}{(h_0,h_1,\ldots,h_{i-1})} \to \sQ{R}{(h_0,h_1,\ldots,h_{i-1},h)}
\]
splits.
Therefore, we obtain $h_i \le h$, as desired.

In particular, $R$ admits a multi-height if and only if $R$ is quasi-$F$-split, and in that case the multi-height is clearly unique.
\end{remark}

\begin{notation}\label{notation:compare}
Let $(R,\m)$ be a Noetherian local $\Z_{(p)}$-algebra such that $R$ is $p$-torsion free and $p$-adically complete, $p \in \m$, and $\var{R} := R/pR$ is $F$-finite.
We set $d := \dim \var{R}$.
\end{notation}

\begin{theorem}\label{mult-ht-to-ppt}
We use \cref{notation:compare}, and assume that $R$ is a quasi-$F$-split complete intersection local ring.
If $R$ has multi-height $\bh = (h_0,h_1,\ldots)$, then
\[
\ppt(R,p) = \sum_{m \geq 1} \frac{c_m}{p^m},
\]
where
\[
c_m :=
\begin{cases}
    p - 1 & \text{if } m = \sum_{i=0}^r h_i \text{ for some } r \geq 0, \\
    p - 2 & \text{otherwise}.
\end{cases}
\]
\end{theorem}

\begin{proof}
It follows from \cite{Yoshikawa25}*{Theorem~5.9}.
\end{proof}

\begin{notation}\label{notation:Fedder}
Let $(A,\m)$ be a regular local ring, and let $\phi \colon A \to A$ be a finite ring homomorphism that lifts the Frobenius morphism.
We assume that $p \in \m \setminus \m^2$, and set $\var{A} := A/pA$, so that $\var{A}$ is regular.
Fix a generator $u \in \Hom_{\var{A}}(F_*\var{A}, \var{A})$.
For an ideal $J$ of $A$ containing $p$,  we set
\[
J^{[p^e]}:=(x^{p^e},p \mid x \in J)
\]
for every $e \geq 0$.
We note that $J^{[p^0]}=J$ and $\phi^e(J)A=J^{[p^e]}$ for every $e \geq 0$.

Let $a_1, \ldots, a_c$ be a regular sequence in $A$, and set $I := (a_1, \ldots, a_c)\var{A}$ and $f := a_1 \cdots a_c$.
Assume that $R := A/(a_1, \ldots, a_c)$ is $p$-torsion free, and set $\var{R} := R/pR$.

Define a map $\theta \colon F_*\var{A} \to \var{A}$ by
\[
\theta(F_*a) := u(F_*(\Delta_1(f^{p-1})a)),
\]
where we put
\[
\Delta_1 \colon A \rightarrow A; \quad a \mapsto \frac{a^p- \phi (a)}{p}.
\]

By \cite{Yoshikawa25}*{Theorem~3.7}, we have an isomorphism
\begin{equation}\label{eq:isom}
\Hom_{\var{A}}(\sQ{A}{(n)}, \var{A}) \simeq \Hom_{\var{A}}(F^n_*\var{A}, \var{A}) \oplus \cdots \oplus \Hom_{\var{A}}(F_*\var{A}, \var{A}) 
\overset{(\star_1)}{\simeq} F^n_*\var{A} \oplus \cdots \oplus F_*\var{A},
\end{equation}
where $(\star_1)$ is given by the natural identification
\[
F^e_*\var{A} \to \Hom_{\var{A}}(F^e_*\var{A}, \var{A});\quad F^e_*a \mapsto (F^e_*b \mapsto u^e(F^e_*(ab)))
\]
for each $n \geq e \geq 1$.

For $g_1, \ldots, g_n \in \var{A}$, we denote by $\varphi_{(g_1, \ldots, g_n)}$ the homomorphism corresponding to the tuple $(F^n_*g_1, \ldots, F_*g_n)$ under the above identification.

Furthermore, for $n_0, \ldots, n_r \in \Z_{\geq 1}$, we define an ideal $I_{(n_0, \ldots, n_r)} \subseteq \var{A}$ inductively as follows:

First, define $I_1 := f^{p-1}\var{A} + I^{[p]}$, and for $n \geq 2$, define
\[
I_n := \theta(F_*I_{n-1}) + I_1.
\]

Next, assume that $I_{(n_1, \ldots, n_r)}$ has been defined for some $n_1, \ldots, n_r \in \Z_{\geq 1}$.
Then, define
\begin{equation}
\label{eqn:defn1}
I_{(1, n_1, \ldots, n_r)} := u(F_*I_{(n_1, \ldots, n_r)}) f^{p-1} + I^{[p]},
\end{equation}
and for $n \geq 2$, define
\begin{equation}
\label{eqn:defn12}
I_{(n, n_1, \ldots, n_r)} := \theta(F_*I_{(n-1, n_1, \ldots, n_r)}) + I_1.
\end{equation}
\end{notation}

\begin{theorem}\label{multi-fedder}
We use \cref{notation:Fedder}.
Then for $n_0,\ldots,n_r \in \Z_{\geq 1}$, we have
\[
\Im\left(\Hom_{\var{R}}(\sQ{R}{(n_0,\ldots,n_r)}, \var{R}) \xrightarrow{\mathrm{ev}} \var{R}\right) = u(F_*I_{(n_0,\ldots,n_r)})\var{R}.
\]
In particular,  the multi-height of $R$ is $(h_0, h_1, \ldots)$ if and only if  for every $r \in \Z_{\geq 0}$ we have
\[
h_r = \inf \left\{ h \mid I_{(h_0, \ldots, h_{r-1}, h)} \nsubseteq \m^{[p]}\var{A} \right\}.
\]
\end{theorem}

\begin{proof}
We prove the assertion by induction on $r$.

For $r=0$, the claim follows from the proof of \cite{Yoshikawa25}*{Theorem~4.13}.

Set $\bn := (n_0,\ldots,n_r)$, $\bn' := (n_1,\ldots,n_r)$, and $n := n_0$.
By definition, we obtain the following pushout diagram:
\[
\begin{tikzcd}[column sep=1.5cm]
    F^{n}_*\var{R} \arrow[r,"F^{n}_*\Phi_{R,\bn'}"] \arrow[d,"V^{n-1}"] & F^{n}_*\sQ{R}{\bn'} \arrow[d,"V^{n-1}"] \\
    \sQ{R}{(n)} \arrow[r] & \sQ{R}{\bn}. \arrow[ul, phantom, "\ulcorner", very near start]
\end{tikzcd}
\]
Take an $\var{R}$-module homomorphism $\psi \colon \sQ{R}{\bn} \to \var{R}$.
Then $\psi$ corresponds to a pair of homomorphisms $\psi_1 \colon \sQ{R}{n} \to \var{R}$ and $\psi_2 \colon F^n_*\sQ{R}{\bn'} \to \var{R}$ satisfying
\[
\psi_1 \circ V^{n-1} = \psi_2 \circ F^n_*\Phi_{R,\bn'}.
\]
Note that
\[
\mathrm{ev}(\psi) = \psi(1) = \psi_1(1).
\]

Since $\var{R}$ is Gorenstein, we have
\[
\Hom_{\var{R}}(F^n_*\sQ{R}{\bn'},\var{R}) \simeq F^n_*\Hom_{\var{R}}(\sQ{R}{\bn'},\var{R}).
\]
Thus, if we take the element $\psi_2' \colon \sQ{R}{\bn'} \to \var{R}$ corresponding to $\psi_2$, then we have $\psi_2 = u^n_R \circ F^n_* \psi_2'$, where $u_R \in \Hom_{\var{R}}(F_* \var{R}, \var{R})$ is a generator.
Hence,
\begin{equation}\label{eq:hom}
\psi_1 \circ V^{n-1} = u^n_R \circ F^n_* \psi_2' \circ F^n_*\Phi_{R,\bn'}.
\end{equation}

We set
\[
x := \psi_2' \circ \Phi_{R,\bn'}(1) \in \Im\left( \Hom_{\var{R}}(\sQ{R}{\bn'}, \var{R}) \xrightarrow{\mathrm{ev}} \var{R} \right) \overset{(\star_1)}{=} u(F_*I_{\bn'}) \overline{R},
\]
where $(\star_1)$ follows from the induction hypothesis.
Thus, $\psi_2' \circ \Phi_{R,\bn'} = \cdot x$.
If we take a lift $\widetilde{x} \in u(F_*I_{\bn'})$ of $x$, then the right-hand side of \eqref{eq:hom} is the homomorphism induced by
\[
u^n(F^n_*(h \cdot -)) \in \Hom_{\var{A}}(F^n_*\var{A}, \var{A})
\]
for $h=f^{p^n-1}\widetilde{x} \in f^{p^{n}-1} u(F_*I_{\bn'})$.

On the other hand, $\psi_1$ lifts to the homomorphism 
\[
\varphi := \varphi_{(g_1,\ldots,g_n)} \colon \sQ{A}{n} \to \var{A}
\]
for some $g_1,\ldots,g_n \in \var{A}$ as in \eqref{eq:isom}, the left-hand side of \eqref{eq:hom} is induced by $u^n(F^n_*(g_1 \cdot -))$.
Hence, we obtain $g_1 \equiv h \mod I^{[p^n]}$, and in particular,
\[
g_1 \in u(F_*I_{\bn'}) f^{p^n - 1} + I^{[p^n]}.
\]

Take $h_1, \ldots, h_n \in \var{A}$ as in \cite{Yoshikawa25}*{Lemma~C.3}.
Then $g_1 \equiv f^{p^n - p} h_1 \mod I^{[p^n]}$, and in particular $h_1 \in I_{(1,\bn')}$.

By the proof of \cite{Yoshikawa25}*{Theorem~4.13}, we have
\[
\mathrm{ev}(\psi) = \psi_1(1) = (\var{A} \to \var{R})(u(F_*g_n)) \in u(F_*I_{\bn})\var{R}.
\]
Therefore,
\[
\Im\left( \Hom_{\var{R}}(\sQ{R}{\bn}, \var{R}) \xrightarrow{\mathrm{ev}} \var{R} \right) \subseteq u(F_*I_{\bn})\var{R}.
\]

We now prove the reverse inclusion.
Take $g \in I_{\bn}$.
By \cite{Yoshikawa25}*{Claim~C.4}, there exist $g_1,\ldots,g_n \in \var{A}$ such that $\varphi := \varphi_{(g_1,\ldots,g_n)}$ induces a homomorphism $\psi_1 \colon \sQ{R}{n} \to \var{R}$ with $g_n \equiv g \mod I_{n-1}$, and $h_1 \in u(F_*I_{\bn'})f^{p-1}$, where $h_1,\ldots,h_n \in \var{A}$ are as in \cite{Yoshikawa25}*{Lemma~C.3}.

In particular, there exists $x \in I_{\bn'}$ such that $h_1 = u(F_*x) f^{p-1}$.
By the induction hypothesis, there exists $\psi_2' \colon \sQ{R}{\bn'} \to \var{R}$ such that $x = \psi_2' \circ \Phi_{R, \bn'}(1)$.
Set $\psi_2 := u^n_R \circ F^n_* \psi_2'$, and then
\[
\psi_1 \circ V^{n-1} = \psi_2 \circ F^n_* \Phi_{R,\bn'}.
\]

Thus, $\psi_1$ and $\psi_2$ glue to give a homomorphism $\psi \colon \sQ{R}{\bn} \to \var{R}$ satisfying $\psi(1) = \psi_1(1)$, and hence
\[
u(F_*g_n) = \psi_1(1) = \mathrm{ev}(\psi) \in \Im\left( \Hom_{\var{R}}(\sQ{R}{\bn}, \var{R}) \xrightarrow{\mathrm{ev}} \var{R} \right).
\]

Since $g_n \equiv g \mod I_{n-1}$, it remains to show that
\[
u(F_* I_{n-1}) \subseteq \Im\left( \Hom_{\var{R}}(\sQ{R}{\bn}, \var{R}) \xrightarrow{\mathrm{ev}} \var{R} \right).
\]

Take $a \in u(F_* I_{n-1})$, and choose a homomorphism $\varphi' \colon \sQ{R}{(n-1)} \to \var{R}$ such that $\varphi'(1) = a$.
Define
\[
\varphi := (\sQ{R}{(n)} \to \sQ{R}{(n-1)} \xrightarrow{\varphi'} \var{R}),
\]
so that $\varphi \circ V^{n-1} = 0$.
Thus, $\varphi$ induces a homomorphism $\psi \colon\sQ{R}{\bn} \to \var{R}$, and hence
\[
a = \psi(1) \in \Im\left( \Hom_{\var{R}}(\sQ{R}{\bn}, \var{R}) \xrightarrow{\mathrm{ev}} \var{R} \right),
\]
as desired.
\end{proof}

Here we record the relationship between the multi-height defined here and the splitting-order sequence introduced in \cite{Yoshikawa25-cri}.

\begin{proposition}\label{sos-and-multi-height}
Let $(A,\m)$ be a regular local ring with a finite Frobenius lift $\phi$ and $p \in \m$.
Fix the $\delta$-ring structure on $A$ induced by $\phi$.
Let $f \in A$ be such that $p,f$ is a regular sequence.
Let $\boldsymbol{s}(f)=(s_i)_{i \geq 0}$ be the splitting-order sequence of $f$ defined in \cite{Yoshikawa25-cri}*{Definition~3.7}.
Assume that $A/(f)$ has the quasi-$F$-split multi-height $\bh=(h_i)_{i \geq 0}$.
Then, for $i \geq 1$, we have
\[
s_i=
\begin{cases}
    0 & \text{if $i=\sum_{j=0}^r h_j$ for some $r \geq 0$}, \\
    1 & \text{otherwise}.
\end{cases}
\]
\end{proposition}

\begin{proof}
For each $r \geq 1$, set
\[
\bolde_r:=\underbrace{(1,\ldots,1}_{\text{$(r-1)$ times}},0).
\]
Then the ideal $I(\bolde_{h_1},\ldots,\bolde_{h_r},\bolde_h)$ defined in \cite{Yoshikawa25-cri}*{Notation~5.1}
coincides with the ideal $I_{(h_1,\ldots,h_r,h)}$ defined in \cref{notation:Fedder} for a regular element $f$.
Therefore, by \cref{multi-fedder} and \cite{Yoshikawa25-cri}*{Theorem~5.2}, we obtain the desired assertion.
\end{proof}

\subsection{Variants of Fedder-type criterion}

In this subsection, we present several variants of the Fedder-type criterion that are practically convenient for computing the multi-height. As the first variant, we show that, in \Cref{multi-fedder}, the ideal $I^{[p]}$  in (\ref{eqn:defn1}) and the ideal $I_1$  in (\ref{eqn:defn12}) are essentially irrelevant.

\begin{proposition}\label{non-ideal-ver}
We use the notation from \cref{notation:Fedder}.
Set $f_1 := f^{p-1}$ and define inductively
\[
f_i := f^{p-1}\Delta_1(f^{p-1})^{1 + p + \cdots + p^{i-2}}
\quad \text{for } i \ge 2.
\]
Then for any integers $n_1,\ldots,n_r \ge 1$, we have
\[
I_{(n_1,\ldots,n_r)}
= u^{n_1+\cdots+n_r-1}
  \Bigl(
    F^{n_1+\cdots+n_r-1}_*
    (f_{n_1}^{p^{n_2+\cdots+n_r}}
     \cdots
     f_{n_{r-1}}^{p^{n_r}}
     f_{n_r}\var{A})
  \Bigr)
  + I_{(n_1,\ldots,n_{r-1},n_r-1)},
\]
where we set $I_0 := I^{[p]}$ and define inductively
\[
I_{(n_1,\ldots,n_{r-1},0)} := I_{(n_1,\ldots,n_{r-1}-1)}.
\]
In particular, if $(h_0,h_1,\ldots)$ is the multi-height of $R$, then
\begin{equation}
\label{eqn:hrnon-ideal}
h_r = \inf\Bigl\{
  h \Bigm|
  f_{h_0}^{p^{h_1+\cdots+h_{r-1}+h}}
  \cdots f_{h_{r-1}}^{p^h}f_h
  \notin \m^{[p^{h_0+\cdots+h_{r-1}+h}]}
\Bigr\}.
\end{equation}

Furthermore, we set a sequence $\{J_r\}$ of ideals of $A$ inductively as follows:
\begin{itemize}
    \item $J_0=\m$.
    \item $J_{r+1}:=(J_r^{[p^{h_r}]}\colon f_{h_r})$.
\end{itemize}
Then we have
\[
h_r=\inf\{h \mid f_h \notin J_{r}^{[p^h]}\}.
\]
\end{proposition}

\begin{proof}
Note that the equation (\ref{eqn:defn1}) holds even if $n_1, \ldots, n_{r-1} \geq 1$ and $n_r =0$. 
Indeed, if $n_i \geq 2$ for some $r-1 \geq i \geq 1$, then this is clear.
When $n_1=\cdots=n_{r-1}=1$, we have
\[
I_{(1,1,\ldots,1,0)} = I_{0} = I^{[p]} =  u(F_* I^{[p]}) f^{p-1} + I^{[p]}
\]
since $If^{p-1} \subseteq I^{[p]}$.
Moreover, the equation (\ref{eqn:defn12}) holds even if $n \geq 2$, $n_1 , \ldots, n_{r-1} \geq 1$ and $n_r =0$.
When $n_i \geq 2$ for some $r-1 \geq i \geq 1$, this is clear.
When $n_1=\cdots=n_{r-1}=1$, we have
\[
I_{(2,1,1,\ldots,1,0)}=I_1=\theta(F_*I^{[p]})+I_1.
\]
Indeed, since
\[
\Delta_1(f^{p-1})a_i^p \equiv \Delta_1(f^{p-1}a_i)-f^{p^2-p}\Delta_1(a_i) \equiv a_i^{p^2}\Delta_1((f/a_i)^{p-1})-f^{p^2-p}\Delta_1(f) \in I_1^{[p]} \pmod{p},
\]
it follows that
\[
\theta(F_*I^{[p]})=u(F_*(\Delta_1(f^{p-1})I^{[p]})) \subseteq I_1.
\]
We first note that $I^{[p]} \subseteq I_{(l_1,\ldots,l_m)}$ for all integers $l_1,\ldots,l_m \ge 0$.
We proceed by induction on $r$.

For $r = 1$, we have
\[
I_1 = f_1 \var{A} + I_0,
\]
and hence inductively,
\[
I_n = \theta(F_* I_{n-1}) + I_1
     = u^{n-1}(F^{n-1}_* (f_n\var{A})) + I_{n-1}
\]
for $n \geq 2$.

Assume now $r \ge 2$.
Then
\[
\begin{aligned}
&I_{(1,n_2,\ldots,n_r)}
= f^{p-1}u(F_* I_{(n_2,\ldots,n_r)}) + I^{[p]} \\
&= f^{p-1}u^{n_2+\cdots+n_r}
  \Bigl(
    F^{n_2+\cdots+n_r}_*
    \bigl(
      f_{n_2}^{p^{n_3+\cdots+n_r}}
      \cdots
      f_{n_{r-1}}^{p^{n_r}}
      f_{n_r}\var{A}
    \bigr)
  \Bigr)  + f^{p-1}u(F_* I_{(n_2,\ldots,n_{r-1},n_r-1)}) + I^{[p]} \\
&= u^{n_2+\cdots+n_r}
  \Bigl(
    F^{n_2+\cdots+n_r}_*
    \bigl(
      f_1^{p^{n_2+\cdots+n_r}}
      f_{n_2}^{p^{n_3+\cdots+n_r}}
      \cdots
      f_{n_{r-1}}^{p^{n_r}}
      f_{n_r}\var{A}
    \bigr)
  \Bigr)
  + I_{(1,n_2,\ldots,n_{r-1},n_r-1)}.
\end{aligned}
\]

Furthermore, inductively,
\[
\begin{aligned}
I_{(n_1,n_2,\ldots,n_r)}
&= \theta(F_* I_{(n_1-1,n_2,\ldots,n_r)}) + I_1\\
&= 
u^{n_1+\cdots+n_r-1}
  \Bigl(
    F^{n_1+\cdots+n_r-1}_*
    \bigl(
      f_{n_1}^{p^{n_2+\cdots+n_r}}
      \cdots
      f_{n_{r-1}}^{p^{n_r}}
      f_{n_r}\var{A}
    \bigr)
  \Bigr) +  \theta (F_* I_{(n_1-1, n_2, \ldots, n_{r-1}, n_r-1)}) + I_1\\
&= u^{n_1+\cdots+n_r-1}
  \Bigl(
    F^{n_1+\cdots+n_r-1}_*
    \bigl(
      f_{n_1}^{p^{n_2+\cdots+n_r}}
      \cdots
      f_{n_{r-1}}^{p^{n_r}}
      f_{n_r}\var{A}
    \bigr)
  \Bigr)
  + I_{(n_1,\ldots,n_{r-1},n_r-1)}
\end{aligned}
\]
for $n_1 \geq 2$.
Then the equation (\ref{eqn:hrnon-ideal}) follows from
\cref{multi-fedder}.

Finally, we take a sequence of ideals $\{J_r\}$ as in the assertion.
We note that since $\var{A}$ is regular, the Frobenius lift $\phi$ is faithfully flat, thus for an ideal $\mathfrak{a}$ of $A$ with $p \in \mathfrak{a}$ and $g \in A$, we have
\begin{align*}
    (\mathfrak{a} \colon g)^{[p]}
    &=(\mathfrak{a} \colon g)\phi_*A =
    (\mathfrak{a} \colon (g,p))\phi_*A \\
    &=(\mathfrak{a} (\phi_*A) \colon (\phi(g),p))=(\mathfrak{a}^{[p]} \colon (g^p,p)) \\
    &=(\mathfrak{a}^{[p]} \colon g^p).
\end{align*}
We prove
\[
(\m^{[p^{h_0+\cdots+h_{r}}]}\colon (f_{h_0}^{p^{h_1+\cdots+h_{r}}}
  \cdots f_{h_{r}}))=J_{r+1}
\]
by induction on $r \geq 0$.
For $r=0$, it follows from definition.
For $r \geq 1$, we have
\begin{align*}
(\m^{[p^{h_0+\cdots+h_{r}}]} \colon 
   (f_{h_0}^{p^{h_1+\cdots+h_{r}}} \cdots f_{h_{r}}))
&= ((\m^{[p^{h_0+\cdots+h_{r-1}}]} \colon 
     (f_{h_0}^{p^{h_1+\cdots+h_{r-1}}} \cdots f_{h_{r-1}}))^{[p^{h_r}]} 
     \colon (f_{h_r})) \\
&= (J_{r}^{[p^{h_r}]} \colon (f_{h_r})) = J_{r+1}.
\end{align*}
by the induction hypothesis.
Therefore, we have
\[
 (\m^{[p^{h_0+\cdots+h_{r-1}+h}]} \colon f_{h_0}^{p^{h_1+\cdots+h_{r-1}+h}}
  \cdots f_{h_{r-1}}^{p^h})=J_{r}^{[p^h]},
\]
as desired.
\end{proof}

\begin{remark}\label{remark:using-element}
The condition
\[
f_{h_0}^{p^{h_1+\cdots+h_{r-1}+h}}
  \cdots f_{h_{r-1}}^{p^h}f_h
  \notin \m^{[p^{h_0+\cdots+h_{r-1}+h}]}
\]
is equivalent to the existence of an element $a \in A$ such that
\[
u^{h_0+\cdots+h_{r-1}+h}\Bigl(F_*^{h_0+\cdots+h_{r-1}+h}\bigl(af_{h_0}^{p^{h_1+\cdots+h_{r-1}+h}}
  \cdots f_{h_{r-1}}^{p^h}f_h\bigr)\Bigr)=1.
\]
Furthermore, since
\[
u^h\bigl(F^h_*(af_h)\bigr)=u\circ \theta^{h-1}\bigl(F^h_*(af^{p-1})\bigr),
\]
if we set $a_r:=u \circ \theta^{h-1}\bigl(F^h_*(af^{p-1})\bigr)$ and define
\[
a_{i}:=u \circ \theta^{h_{i}-1}\bigl(F^{h_i}_*(a_{i+1}f^{p-1})\bigr)
\]
inductively for $r-1 \geq i \geq 0$, then we obtain
\[
1=u^{h_0+\cdots+h_{r-1}+h}\Bigl(F_*^{h_0+\cdots+h_{r-1}+h}\bigl(af_{h_0}^{p^{h_1+\cdots+h_{r-1}+h}}
  \cdots f_{h_{r-1}}^{p^h}f_h\bigr)\Bigr)=a_0.
\]
\end{remark}

The following proposition provides a criterion giving a sufficient condition for the multi-height to become $1$ from some point on.

\begin{proposition}\label{stable-1}
We use the notation of \cref{notation:Fedder}.
Let $J \subseteq \var{A}$ be an ideal and $n_0, \ldots, n_r \in \Z_{\ge 1}$.
We define ideals $I^J_{(n_0, \ldots, n_r)} \subseteq \var{A}$ inductively as follows.

First, set
\[
I^J_1 := f^{p-1}J + I^{[p]},
\qquad
I^J_n := \theta(F_*I^J_{n-1}) + I_1 \quad (n \ge 2).
\]
Next, suppose that $I^J_{(n_1, \ldots, n_r)}$ has been defined for some $n_1, \ldots, n_r \in \Z_{\ge 1}$.
Then define
\[
I^J_{(1, n_1, \ldots, n_r)} := u(F_*I^J_{(n_1, \ldots, n_r)}) f^{p-1} + I^{[p]},
\qquad
I^J_{(n, n_1, \ldots, n_r)} := \theta(F_*I^J_{(n-1, n_1, \ldots, n_r)}) + I_1 \quad (n \ge 2).
\]

Assume further that $J \subseteq u(F_*I^J_1)$ and the multi-height of $R$ is $(h_0,h_1,\ldots)$.
\begin{enumerate}
\item Then we have
\[
I^J_{(n_0,\ldots,n_r)} \subseteq I^J_{(n_0,\ldots,n_r,1)}
\]
for every $n_0, n_1, \ldots, n_r \in \Z_{\ge 1}$.
Consequently, if  
$I^J_{(h_0,\ldots,h_r)} \nsubseteq \m^{[p]}\var{A}$, then $h_i=1$ for every $i \ge r+1$.
\item If there exists $a \in J$ such that
\[
af_{h_0}^{p^{h_1+\cdots+h_{r-1}+h_r}}
  \cdots f_{h_{r-1}}^{p^{h_r}}f_{h_r}
  \notin \m^{[p^{h_0+\cdots+h_{r-1}+h_r}]},
\]
then $h_i=1$ for every $i \ge r+1$.
\end{enumerate}
\end{proposition}

\begin{proof}
We prove the first assertion in (1) by induction on $r$.

If $r=0$, then
\[
I^J_{(1,1)}
= u(F_*I^J_1)f^{p-1}+I^{[p]}
\supseteq Jf^{p-1}+I^{[p]}=I_1^J.
\]
Hence, by induction on $n_0$, we obtain 
$I^J_{n_0} \subseteq I^J_{(n_0,1)}$, as desired.

Now assume $r \ge 1$.
Then
\[
I^J_{(1,n_1,\ldots,n_r,1)}
= u(F_*I^J_{(n_1,\ldots,n_r,1)})f^{p-1}+I^{[p]}
\supseteq u(F_*I^J_{(n_1,\ldots,n_r)})f^{p-1}+I^{[p]}
= I^J_{(1,n_1,\ldots,n_r)},
\]
where we use the induction hypothesis in the middle inclusion.
Therefore, by the induction on $n_0$, we obtain
\[
I^J_{(n_0,n_1,\ldots,n_r)} 
\subseteq I^J_{(n_0,n_1,\ldots,n_r,1)},
\]
as claimed.

Next, suppose that  
$I^J_{(h_0,\ldots,h_r)} \nsubseteq \m^{[p]}\var{A}$.
Then
\[
I_{(h_0,\ldots,h_r,1,\ldots,1)}  \supseteq 
I^J_{(h_0,\ldots,h_r,1,\ldots,1)} 
\supseteq I^J_{(h_0,\ldots,h_r)} 
\nsubseteq \m^{[p]}\var{A}.
\]
Hence $I_{(h_0,\ldots,h_r,1,\ldots,1)} \nsubseteq \m^{[p]}\var{A}$,
and by \Cref{multi-fedder}, we must have 
$h_i = 1$ for all $i \ge r+1$, as desired.

Finally, we prove (2).
By assumption, we have
\[
u^{h_0+\cdots+h_{r-1}+h_r}(F^{h_0+\cdots+h_{r-1}+h_r}_*(af_{h_0}^{p^{h_1+\cdots+h_{r-1}+h_r}}
  \cdots f_{h_{r-1}}^{p^{h_r}}f_{h_r}\var{A})) \nsubseteq \m^{[p]}.
\]
By the proof of \cref{non-ideal-ver} and \cref{remark:using-element}, the left-hand side is contained in $I^J_{(h_0,\ldots,h_r)}$, thus the assertion follows from (1).
\end{proof}

Finally, we recall a sufficient condition, formulated in terms of the positive-characteristic quasi-$F^{\infty}$-height, for having $h_i=1$ for all $i\ge 1$.

\begin{proposition}\label{ht=ht^infty}
We use the notation of \cref{notation:Fedder}.
If $h:=\sht(R)=\sht^\infty(\var{R})$, then the multi-height of $R$ is $(h,1,\ldots)$.
\end{proposition}

\begin{proof}
It follows from \cite{Yoshikawa25}*{Proposition~4.6}.
\end{proof}

\subsection{Naive multi-height}

Originally, the quasi-$F$-split height was defined for schemes in positive characteristic, rather than in mixed characteristic. The height in positive characteristic is also important for computing the mixed-characteristic height, in that it provides an upper bound for the mixed-characteristic height (independent of the choice of a lift). For this reason, it would be desirable to define the multi-height in positive characteristic as well; unfortunately, we do not currently have an intrinsic definition.

In this subsection, by starting from the Fedder-type criterion \Cref{multi-fedder}, we introduce, in the complete intersection case, a sequence of integers (the \emph{naive multi-height}) that provides an upper bound for the mixed-characteristic height.

\begin{proposition}\label{multi-height-positive-char}
We use the notation of \cref{notation:Fedder}.
We define ideals $\var{I}_{(n_0, \ldots, n_r)} \subseteq \var{A}$ inductively as follows.

First, set
\[
\var{I}_1 := f^{p-1}\var{A} + I^{[p]}, \qquad
\var{I}_n := \theta\bigl(F_*(\var{I}_{n-1} \cap \Ker(u))\bigr) + I_1 
\quad (n \ge 2).
\]
Next, suppose that $\var{I}_{(n_1, \ldots, n_r)}$ has been defined for some $n_1, \ldots, n_r \in \Z_{\ge 1}$.
Then define
\[
\begin{aligned}
\var{I}_{(1, n_1, \ldots, n_r)} 
&:= u(F_* \var{I}_{(n_1, \ldots, n_r)}) f^{p-1} + I^{[p]}, \\[4pt]
\var{I}_{(n, n_1, \ldots, n_r)} 
&:= \theta\bigl(F_*(\var{I}_{(n-1, n_1, \ldots, n_r)} \cap \Ker(u))\bigr)
    + I_1, \qquad (n \ge 2).
\end{aligned}
\]

For all $n_1,\ldots,n_r \in \Z_{\ge 1}$, we have
\[
\var{I}_{(n_1,\ldots,n_r)} \subseteq I_{(n_1,\ldots,n_r)},
\]
and the ideals $\var{I}_{(n_1,\ldots,n_r)}$ are invariant under replacing the regular sequence
$a_1,\ldots,a_c$ by any other regular sequence $a_1',\ldots,a_c'$
satisfying $a_i' \equiv a_i \pmod{p}$ for every $1 \le i \le c$.
\end{proposition}

\begin{proof}
The inclusion $\var{I}_{(n_1,\ldots,n_r)} \subseteq I_{(n_1,\ldots,n_r)}$ follows directly from the definition.
Let $f' \in A$ satisfy \(f \equiv f' \pmod{p}\).
Then we have
\[
\Delta_1(f) \equiv \Delta_1(f') \pmod{F(\var{A})}.
\]
By induction on the construction of $\var{I}_{(n_1,\ldots,n_r)}$, 
it follows that $\var{I}_{(n_1,\ldots,n_r)}$ depends only on the reduction of \(f\) modulo \(p\),
as claimed.
\end{proof}

\begin{definition}
\label{defn:naive_heights}
We use the same notation as in \cref{multi-height-positive-char}.
Assume that $\overline{R} = \overline{A}/\overline{f}$ is quasi-$F$-split of height $h$.
We define the \emph{naive multi-height of $\var{R}$ with respect to $(A,\phi,  \var{f})$} as a sequence $\boldsymbol{m} = (m_i)_{i \geq 0} \in \prod \Z_{\geq 1}$ determined by the equalities
\[
m_r = \inf \left\{ m \mid \var{I}_{(m_0, \ldots, m_{r-1}, m)} \nsubseteq \m^{[p]}\var{A} \right\}.
\]
for any $r \geq 0$.
Note that the infimum appearing above is finite and less than or equal to $h$.
Indeed, since $\var{I}_h \nsubseteq \m^{[p]}$, we have
\[
\var{I}_{(1,h)}=u(F_*\var{I}_h)f^{p-1}+I^{[p]}=f^{p-1}\var{A}+I^{[p]}=\var{I}_1,
\]
and in particular, we have $\var{I}_{(m_0,\ldots,m_{i-1},h)} =  \var{I}_{(m_0,\ldots,m_{i-1})}$.
Note that this definition may depend on the choice of lifts of Frobenius $\phi$. 
Let $(h_i)_{i \geq 0}$ be the multi-height of $R$.
Then \cref{multi-height-positive-char} shows that $(m_i)$ is greater than or equal to $(h_i)$ with respect to the lexicographical order. 
\end{definition}

\begin{remark}
\label{rem:naive_first}
The first term $m_0$ of the naive multi-height $\boldsymbol{m}$ of $\var{R}$ is equal to the quasi-$F$-split height of $\var{R}$ by the Fedder-type criteria \cite{kty}*{Theorem~4.11}.
In particular, $m_0$ depends only on $\var{R}$.
\end{remark}

\subsection{Pre-periodic}
The multi-height of a Noetherian $\Z_{(p)}$-algebra is conjectured to be pre-periodic (, provided that it is quasi-$F$-split.).
Note that the pre-periodicity of multi-height is equivalent to $\ppt(-,p) \in \Q$ in this setting.
If $\bn =  (n_0, n_1, \ldots)\in \prod \Z_{\geq 1}$ 
is pre-periodic, i.e., there exists $r \in \Z_{\geq 0}$ and $s \in \Z_{>0}$ such that $n_{r+i} = n_{r +i+s}$ for any $i\geq 0$,
then we write
\[
\bn = (n_0, \ldots, n_{r-1}, \overline{n_{r}, n_{r+1}, \cdots, n_{r+s-1} }) 
\]
as a repeating decimal.

The following proposition is useful for computing the multi-height by using a computer algebra system.
\begin{prop}
\label{prop:repeating_multiheight}
We use \cref{notation:Fedder}.
Let 
\[
\bn = (n_i)_{i\geq0}=  (n_0, \ldots, n_{r-1}, \overline{n_{r}, n_{r+1}, \cdots, n_{r+s-1} }) \in \prod \Z_{\geq 1}
\]
be a pre-periodic sequence.
Assume that there exists a positive integer $t$ that satisfies the following:
\begin{enumerate}
    \item 
For every $0 \leq m\leq r+ts-1$, we have
\begin{equation}
\label{eqn:nm}
n_m = \inf \{ n \mid I_{(n_0, \ldots, n_{m-1}, n)} \not\subset \m^{[p]}\var{A} \}.
\end{equation}
\item For every $0 \leq u \leq s-1$, we have
\begin{eqnarray*}
&&I_{(\underbrace{\scriptstyle n_r, ..., n_{r+s-1}, \ldots,  n_r, \ldots, n_{r+s-1}}_{\text{$t-1$ times}}, n_{r}, \ldots, n_{r+u})} \\
&=&I_{(\underbrace{\scriptstyle n_r, ..., n_{r+s-1}, \ldots,  n_r, \ldots, n_{r+s-1}}_{\text{$t$ times}}, n_{r}, \ldots, n_{r+u})}.
\end{eqnarray*}
\item 
For every $0 \leq u \leq s-1$ satisfying $n_{r+u} \geq 2$, we have
\begin{eqnarray*}
&&I_{(\underbrace{\scriptstyle n_r, ..., n_{r+s-1}, \ldots,  n_r, \ldots, n_{r+s-1}}_{\text{$t-1$ times}}, n_{r}, \ldots, n_{r+u-1}, n_{r+u}-1)} \\
&=&I_{(\underbrace{\scriptstyle n_r, ..., n_{r+s-1}, \ldots,  n_r, \ldots, n_{r+s-1}}_{\text{$t$ times}}, n_{r}, \ldots, n_{r+u-1}, n_{r+u}-1)}.
\end{eqnarray*}
\end{enumerate}
Then the multi-height of $R$ is equal to $\bn$.
\end{prop}

\begin{proof}
By \cref{multi-fedder}, it suffices to verify \eqref{eqn:nm} for every non-negative integer $m$.
By (1), we may assume that $m > r+ts-1$.
We may write
\[
m = r + vs + u
\]
for $v\geq t$ and some $0 \leq u \leq s-1$.

By (2) and the definition of $I_{(n_{1}, \ldots, n_m)}$, we obtain
\begin{eqnarray*}
I_{(n_0, n_1, \ldots, n_{r+(t-1)s+u})}
&=&I_{(n_0, n_1, \ldots, n_{r-1},\underbrace{\scriptstyle n_r, ..., n_{r+s-1}, \ldots,  n_r, \ldots, n_{r+s-1}}_{\text{$t-1$ times}}, n_r ,\ldots, n_{r+u})} \\
&=&I_{(n_0, n_1, \ldots, n_{r-1},\underbrace{\scriptstyle n_r, ..., n_{r+s-1}, \ldots,  n_r, \ldots, n_{r+s-1}}_{\text{$t$ times}}, n_r ,\ldots, n_{r+u})} \\
&=& I_{(n_0, n_1, \ldots, n_{r-1},\underbrace{\scriptstyle n_r, ..., n_{r+s-1}, \ldots,  n_r, \ldots, n_{r+s-1}}_{\text{$v$ times}}, n_r ,\ldots, n_{r+u})} 
=I_{(n_0, \ldots, n_m)}.
\end{eqnarray*}
Combining with (1), we obtain $I_{(n_0, \ldots, n_m)} \not\subset \m^{[p]}\var{A}$, i.e., we have 
\[
n_m \geq \inf \{ n \mid I_{(n_0, \ldots, n_{m-1}, n)} \not\subset \m^{[p]}\var{A} \}.
\]
We now show the reverse inequality.
We may assume $n_m \geq 2.$
It suffices to show $I_{(n_0, \ldots, n_{m-1}, n_m-1)} \subset \m^{[p]}\var{A}$.
By (3) and the definition of $I_{(n_1, \ldots, n_m)}$, we obtain
\begin{eqnarray*}
I_{(n_0, n_1, \ldots, n_{r+(t-1)s-1+u},  n_{r+(t-1)s+u}-1)}
&=&I_{(n_0, n_1, \ldots, n_{r-1},\underbrace{\scriptstyle n_r, ..., n_{r+s-1}, \ldots,  n_r, \ldots, n_{r+s-1}}_{\text{$t-1$ times}}, n_r ,\ldots, n_{r+u-1}, n_{r+u}-1)} \\
&=&I_{(n_0, n_1, \ldots, n_{r-1},\underbrace{\scriptstyle n_r, ..., n_{r+s-1}, \ldots,  n_r, \ldots, n_{r+s-1}}_{\text{$t$ times}}, n_r ,\ldots, n_{r+u-1}, n_{r+u}-1)} \\
&=& I_{(n_0, n_1, \ldots, n_{r-1},\underbrace{\scriptstyle n_r, ..., n_{r+s-1}, \ldots,  n_r, \ldots, n_{r+s-1}}_{\text{$v$ times}}, n_r ,\ldots, n_{r+u-1}, n_{r+u}-1)} \\
&=&I_{(n_0, \ldots, n_{m-1}, n_m-1)}.
\end{eqnarray*}
Combining with (1), we obtain $I_{(n_0, \ldots, n_{m-1}, n_{m}-1)} \subset \m^{[p]}\var{A}$ as desired.
\end{proof}

\begin{remark}
\label{rem:repeating_naive_multiheight}
A similar proposition to \Cref{prop:repeating_multiheight} also holds for the naive multi-height, and the proof is entirely analogous.
\end{remark}

\section{List of perfectoid pure thresholds of rational double points}

In this section, we discuss $\ppt (-,p)$ for lifts of rational double points (RDPs).
To state our results, we first clarify the meaning of the terms "RDPs" and "lifts".
For simplicity, we use these terms only in the setting of complete local rings:
\begin{definition}
Let $k$ be an algebraically closed field of characteristic $p>0$, and let \(\var{R}\) be a complete local ring over \(k\).
\begin{enumerate}
\item 
We say that \(\var{R}\) is an \emph{RDP} if it is isomorphic to the completed local ring of a non-smooth canonical surface singularity over \(k\).
Equivalently, by the classification in \cite[Section 3]{Art}, $\var{R}$ is isomorphic to the complete local ring defined by one of the equations listed there.
\item
A \emph{$W(k)$-lift} of $\var{R}$ is a complete local $W(k)$-flat algebra $R$ such that $R /pR \simeq \var{R}.$
\end{enumerate}
\end{definition}

\begin{remark}
More precisely, the notion of a \(W(k)\)-lift used here should more precisely be called a \emph{complete local} \(W(k)\)-lift.
However, if there exists a local \(W(k)\)-lift \(R\) of \(\var{R}\), possibly non-complete, then the proof of \cite[Lemma 4.8]{p-pure} shows that
\[
\operatorname{ppt}(R,p)=\operatorname{ppt}(\widehat{R},p).
\]
Therefore, for the purpose of computing \(\operatorname{ppt}(-,p)\), it is enough to restrict attention to the complete case.
\end{remark}


\begin{proposition}\label{lift-RDP}
Let $k$ be an algebraically closed field of characteristic $p>0$.
Let $R$ be a $W(k)$-lift of an RDP over $k$.
Then $R \simeq W(k)[[x,y,z]]/(f)$, where $f$ is a lift of an equation in the list in \cite{Art}*{Section~3}.
\end{proposition}

\begin{proof}
Since $R/p$ is an RDP, there exists $\var{f} \in k[[x,y,z]]$ in the list in \cite{Art}*{Section~3} such that
\[
R/p \simeq k[[x,y,z]]/(\var{f}).
\]
Fix lifts $X,Y,Z$ of $x,y,z$ in $R$, and consider the ring homomorphism
\[
W(k)[x,y,z] \to R, \qquad x,y,z \mapsto X,Y,Z.
\]
Since $R$ is complete and the homomorphism is surjective modulo $p$, it induces a surjective ring homomorphism
\[
\varphi \colon W(k)[[x,y,z]] \to R.
\]

Take a lift $f$ of $\var{f}$ in $W(k)[[x,y,z]]$ with $f \in \Ker(\varphi)$.
Since $\var{f}$ is a regular element, the sequence $p,f$ is a regular sequence, and in particular, the ring $W(k)[[x,y,z]]/(f)$ is $p$-torsion free.
Then the induced ring homomorphism
\[
\varphi' \colon W(k)[[x,y,z]]/(f) \to R
\]
is an isomorphism modulo $p$.
Since $W(k)[[x,y,z]]/(f)$ is $p$-torsion omplete, it follows that $\varphi'$ is an isomorphism.
\end{proof}

In the following of this paper, we freely use \cref{lift-RDP}.




First, we compute the perfectoid pure threshold of lifts of RDPs, excluding those of type $D_{2n}^r$ and $D_{2n+1}^r$ for $1 \leq r \leq n-1$ in characteristic $2$.
\begin{proposition}
Let $k$ be an algebraically closed field of characteristic $p>0$ and let $\var{R} :=k[[x,y,z]]/(\var{f})$ be an RDP.
Then we have the following:
\begin{enumerate}
    \item 
If $\var{R}$ is a taut RDP, then $\ppt (R,p) =1$ for any  $W(k)$-lift $R$ of $\var{R}$.
    \item 
If $\var{R}$ is a non-taut RDP that is neither of type $D_{2n}^r$ nor $D_{2n+1}^r$ ($1 \leq r \leq n-1$), then the set of values taken by $\ppt (R,p)$ for  $W(k)$-lifts $R$ of $\var{R}$ is given in Table \ref{table:ppts}.
\end{enumerate}
\end{proposition}

\begin{proof}
(1) follows from \cite{p-pure}*{Theorem~6.6} since $\var{R}$ is $F$-pure.

We shall prove (2).
By the same proof as in (1), we obtain $\ppt (R,p) = 1$ for any $W(k)$-lift of $E_6^1$, $E_7^3$, and $E_8^4$ type RDPs in characteristic $2$, $E_6^1$, $E_7^1$, $E_8^2$ type RDPs in characteristic $3$, and $E_8^1$ type RDP in characteristic $5$.

Let $R := W(k)[[x,y,z]]/(f)$, where $f$ is the natural lift of $\var{f}$ as in Table \ref{table:ppts}.
In the following, we use \Cref{mult-ht-to-ppt} and Remark \ref{remark:using-element} freely.

In the following, we discuss the case where $p=2$.
We consider the case of type $D_{2n}^0$.
Let $f_G = f + pG \in W(k)[[x,y,z]]$ be a lift of $\var{f}$, where $G \in W(k)[[x,y,z]]$ is an arbitrary element.
Let $(h_i)_{i\geq 0} \in \prod \Z_{\geq 1}$ be the multi-height of $R_G:=W(k)[[x,y,z]]/(f_G)$.
Note that, by \cite[Section 4.7]{KTY25} and \cite[Proposition 4.6]{Yoshikawa25}, we have $1 \leq h_0 \leq  \lceil \log_2 n \rceil +1$.
We assume that $h_0 \neq 1$.
We have $I_{h_0} \not\subset \m^{[p]}$.
Therefore, there exists an element $a_1 \in W(k)[[x,y,z]]$ such that 
\[
u^{h_0-1}(F^{h_0-1}_*(a_1 f_{G} \Delta_1 (f_{G})^{1+ \cdots + 2^{h_0-2}})) \notin \m^{[p]}.
\]
Since the $z$-degree of each term of $f_{G} \Delta_1 (f_{G})^{1+ \cdots + 2^{h_0-2}}$ is even, we have $a_1 \in (z)$.
Therefore, $I_{h_0}^{(z)} \not\subset \m^{[p]}$.
Note that $(z) \subset u ( F_* I_1^{(z)})$.
By \Cref{stable-1}, we have $h_i = 1$ for $i\geq 1$.
Therefore, the only possible values of $\ppt(-,p)$ are of the form $1/2^i$ ($1 \leq i \leq \lceil \log_2 n \rceil$).
One can verify directly that the $G$ listed in Table \ref{table:ppts} yield the desired values of $\ppt(-,p)$.

The case of type $D_{2n+1}^0$ follows in the same way by using $x$ instead of $z$.
The cases of types $E_6^0$, $E_7^0$, $E_8^0$ also follow in the same way.

Other cases are a bit involved.
It can be checked directly, as before, that the values of $\ppt(-,p)$ listed in Table \ref{table:ppts} are realized by the corresponding $G$.
Therefore, in what follows, we present only the argument that reduces the candidates for $\ppt(-,p)$ to those listed in Table \ref{table:ppts}.

We consider the case of the type $E_8^1$.
Let $(h_i)_{i\geq 0} \in \prod \Z_{\geq 1}$ be the multi-height of $R_G:=W(k)[[x,y,z]]/(f_{G})$ as before.
Since $\sht (\var{R})=4$ (\cite[Section 4.7]{KTY25}), we have $2 \leq h_0 \leq 4$ by \cite[Proposition 4.6]{Yoshikawa25}. 
If $h_0=4$, then we have $h_i =1$ for any $i\geq 1$
by \cite{TY26}*{Theorem~A} and Proposition \ref{ht=ht^infty}.

Next, we assume that $h_0 =2$.
Then there exists $a_1 \in W(k)[[x,y,z]]$ such that
\begin{equation}
\label{eqn:E8a1}
\theta (F_* a_1) =  u(F_* ( a_1 f_{G} \Delta_1 (f_{G}))) \notin \m^{[p]}.
\end{equation}
Since the height of $W(k)[[x,y,z]]/(f)$ is $3$, we have
\[
u(F_* (a_1 f G^2 )) \notin \m^{[p]}.
\]
We assume that we can take $a_1 \in \m$.
By taking the monomial decomposition, we may assume that $a_1$ is monomial.
Then clearly, we have $a_1 \in (x,y^2,z)$.
Since $(x,y^2, z) \subset u(F_* I_1^{(x,y^2,z)})$, by \Cref{stable-1}, we have $h_i =1$ for $i \geq 1$ in this case.
Next, we assume that any $a_1$ as above satisfies $a_1 \notin \m$.
In this case, we have $G_{100} \equiv G_{001} \equiv G_{000} \equiv 0 \mod 2$ if we put $G = \sum G_{ijk} x^i y^j z^k$. 
We take a pair $(a_1, a_2)$ as in Remark \ref{remark:using-element} (then $a_1$ satisfies (\ref{eqn:E8a1})).
We have 
\[
a_1 = u( \theta^{h_1-1} (F^{h_1}_* (a_2 f_G^{p-1}) )).
\]
Since $u (F_* (f_{G})) \subset \m$, we have $h_1 \geq 2$.
Therefore, we have $h_1 =2$ by $h_0=2$ and Remark \ref{remark:ineq-h}.
If we may take $a_2 \in \m$, we have 
\[
\theta (F_* (a_2f_G^{p-1})) \in \m^{[p]}
\]
since
$G_{100} \equiv G_{001} \equiv G_{000} \equiv 0 \mod 2$.
Thus we have $a_{1} \in \m$, which contradicts the assumption.
Therefore, we have $a_2 \notin \m$.
Repeating this argument, we obtain $h_i = 2$ for $i\geq 2$.

Finally, we assume that $h_0 =3$.
We put $G = \sum G_{ijk} x^i y^j z^k$.
By direct computation,
\[ 
f_{G} \Delta_1 (f_{G}) \in \m^{[4]}
\]
implies that 
\[
G_{000} \equiv G_{100} \equiv G_{010} \equiv G_{001} \equiv G_{110} \equiv G_{011} \equiv G_{101} \equiv 0 \mod 2.
\]
Moreover, we have
\[
f_{G} \Delta_1 (f_{G})^{1+2} \equiv
G_{021}^2 x^7y^7z^7 + G_{020}^2 x^7y^7 z^5 + G_{030}^2 x^6y^6z^6 + G_{020}^2 x^6 y^4 z^6 \mod \m^{[8]}
\]
Therefore, if $G_{020} \not\equiv 0$ or $G_{030} \not\equiv 0 \mod 2$, then there exists $a_1 \in (x,y,z^2)$ such that 
\[
\theta^2 (F^2_* (a_1f_G)) \notin \m^{[p]}.
\]
By Proposition \ref{prop:repeating_multiheight} again, we have $h_i =1$ for $i\geq 1$ in this case.
On the other hand, if we have $G_{020} \equiv G_{030}  \equiv 0 \mod 2$,
then $G_{021} \not \equiv 0 \mod 2$ since $h_0 =3$.
We take a pair $(a_1, a_2)$ as in Remark \ref{remark:using-element}.
We have $a_1 \notin \m$.
By $h_0 =3$, we have $h_1 =3$ in this case.
Moreover, by the assumption, we have $a_2 \notin \m$.
Repeating this argument, we have $h_i = 3$ for $i\geq 2$.
It finishes the proof of this case.

Next, we consider the case of type $E_7^1$.
Let $(h_i)_{i\geq 0} \in \prod \Z_{\geq 1}$ be the multi-height of $R_G:=W(k)[[x,y,z]]/(f_{G})$ as before.
Since $\sht (\var{R})=3$ (\cite[Section 4.7]{KTY25}), we have $2 \leq h_0 \leq 3$ by \cite[Proposition 4.6]{Yoshikawa25}. 
First, we assume $h_0=2$, then there exists $a_1 \in W(k)[[x,y,z]]$ such that
\begin{equation}
\label{eqn:E7a1}
\theta (F_* a_1) =  u(F_* ( a_1 f_{G} \Delta_1 (f_{G}))) \notin \m^{[p]}.
\end{equation}
A direct computation shows that in this case we may take
\[
a_1 \in (x,y^2,z).
\]
Since
\[
(x,y^2,z) \subset u(F_* I_1^{(x,y^2,z)}),
\]
it follows from \Cref{stable-1} that $h_i=1$ for all $i\geq 1$.

Finally, assume that $h_0=3$.
Put
\[
G=\sum G_{ijk}x^iy^jz^k.
\]
As we have checked above,
\[
f_G\Delta_1(f_G)\in \m^{[4]}
\]
implies
\[
G_{000}\equiv G_{100}\equiv G_{010}\equiv G_{001}\equiv G_{110}\equiv G_{011}\equiv G_{101}\equiv0 \mod 2.
\]
Under this assumption, a direct computation shows that
\[
\begin{aligned}
f_G\Delta_1(f_G)^{1+2}\equiv{}\;&
x^7y^7z^7
+ G_{021}^2 x^7y^7z^6 \\
&+ G_{020}^2
\left(
x^7y^7z^4
+ x^7y^6z^6
+ x^6y^7z^7
+ x^6y^4z^6
\right) \\
&+ (G_{030}^2+G_{200}^2)x^6y^6z^6
\mod \m^{[8]}.
\end{aligned}
\]
Therefore, if
\[
G_{020}\not\equiv 0,\quad\text{or}\quad G_{021}\not\equiv 0,\quad\text{or}\quad
G_{030}\not\equiv G_{200} \mod 2,
\]
then there exists $a_1\in (x,z)$ such that
\[
u^2(F_*^2(a_1f_G\Delta_1(f_G)^{1+2})) \notin \m^{[2]}.
\]
Hence, by \Cref{stable-1}, we obtain $h_i=1$ for all $i\geq 1$.

On the other hand, if
\[
G_{020}\equiv G_{021}\equiv0
\quad\text{and}\quad
G_{030}\equiv G_{200}\mod 2,
\]
then
\[
f_G\Delta_1(f_G)^{1+2}\equiv x^7y^7z^7 \mod \m^{[8]}.
\]
Let $(a_1,a_2)$ be as in Remark \ref{remark:using-element}.
Then we have $a_1\notin \m$.
Since $h_0=3$, it follows that $h_1=3$ in this case.
Moreover, by the same argument, we also have $a_2\notin \m$.
Repeating this argument, we obtain $h_i=3$ for all $i\geq 2$.

The other cases can be proved in the same way.
\end{proof}

\begin{table}[ht]
\begin{threeparttable}
\caption{List of \texorpdfstring{$\ppt(-,p)$}{ppt(p)} of $W(k)$-lifts of non-taut RDPs (except for type $D_{2n}^r$ and $D_{2n+1}^{r}$)}
\label{table:ppts}
\centering
\begin{tabular}{|l|c|c|c|}
\hline
$p$ & type &  $\var{f}$ & \text{list of } $\ppt(-,p)$ \tnote{*}
\\
\hline
2 & $D_{2n}^{0}$ &$z^2 +x^{2}y+xy^n$ &  $1/2^{ \lceil\log_2 n \rceil}$ $(0), 1/2^i$ $(y^{2^{i-1}})$ \text{ for }$1 \leq i < \lceil\log_2 n \rceil$
\\
\hline
2 & $D_{2n+1}^{0}$ &$z^2+x^2y+y^nz$ & 
$1/2^{ \lceil\log_2 n \rceil}$ $(0), 1/2^i$ $(y^{2^{i-1}})$ \text{ for }$1 \leq i < \lceil\log_2 n \rceil$
\\
\hline
2 & $E_{6}^{0}$ &$z^2+x^3+y^2z $& 1/2 (0) \\
\hline
2 & $E_{6}^{1}$ &$z^2+x^3+y^2z+xyz$ & 1 (0) \\
\hline
2 & $E_{7}^{0}$ &$z^2+x^3+xy^3$ & 1/8 (0), 1/4 ($y^2$), 1/2 (1)\\
\hline
2 & $E_{7}^{1}$ &$z^2+x^3+xy^3+x^2yz$& 1/7 (0), 1/4 ($y^2$) ,1/2 (1)\\
\hline
2 & $E_{7}^{2}$ &$z^2+x^3+xy^3+y^3z$ &1/3 (0), 1/2 (1)\\
\hline
2 & $E_{7}^{3}$ &$z^2+x^3+xy^3+xyz$& 1 (0) \\
\hline
2 & $E_{8}^{0}$ &$z^2+x^3+y^5$ & 1/8 (0), 1/4 ($y^2$), 1/2 (1) \\
\hline
2 & $E_{8}^{1}$ &$z^2+x^3+y^5+xy^3z$ & 1/8 (0), 1/7 ($y^2z$), 1/4 ($y^2$), 1/3 ($xz$), 1/2 (1)\\
\hline
2 & $E_{8}^{2}$ &$z^2+x^3+y^5+xy^2z$ & 1/7 (0), 1/4 ($y^2$), 1/2 (1)\\
\hline
2 & $E_{8}^{3}$ &$z^2+x^3+y^5+y^3z$ & 1/3 (0), 1/2 (1)\\
\hline
2 & $E_{8}^{4}$ &$z^2+x^3+y^5+xyz$& 1 (0)\\
\hline
3 & $E_{6}^{0}$ & $z^2+x^3+y^4$ & 2/3 (0)\\ 
\hline
3 & $E_{6}^{1}$ &$z^2+x^3+y^4+x^2y^2$ &1 (0)\\ 
\hline
3 & $E_{7}^{0}$ & $z^2+x^3+xy^3$ &2/3 (0)\\ 
\hline
3 & $E_{7}^{1}$ &$z^2+x^3+xy^3+x^2y^2$ & 1 (0)\\ 
\hline
3 & $E_{8}^{0}$ & $z^2+x^3+y^5$ & 5/9 (0), 2/3 (1)\\ 
\hline
3 & $E_{8}^{1}$ & $z^2+x^3+y^5+x^2y^3$ &5/8 (0), 2/3 (1)\\ 
\hline
3 & $E_{8}^{2}$ & $z^2+x^3+y^5+x^2y^2$ & 1 (0)\\ 
\hline
5 & $E_{8}^{0}$ &$z^2+x^3+y^5$  & 4/5 (0)\\
\hline
5 & $E_{8}^{1}$ &$z^2+x^3+y^5+xy^4$ & 1 (0)\\ 
\hline
\end{tabular}
\begin{tablenotes}
\item[*]
In the rightest column, the parentheses to the right of each value, we indicate a choice of \( G \) such that
\(\ppt( W(k)[[x,y,z]]/( f + pG),p)\) realizes the value.
Here, \(f\) denotes the natural lift of \(\var{f}\) obtained by viewing all coefficients of \(\var{f}\) as elements of \(\mathbb{Z} \subset W(k)\).
\end{tablenotes}
\end{threeparttable}
\end{table}

\section{On the types \texorpdfstring{$D_{2n}^r$}{D(2n)r} and \texorpdfstring{$D_{2n+1}^r$}{D(2n+1)r}}

\begin{lemma}\label{f_h-lemma}
Let $p=2$ and $A=W(k)[[x,y,z]]$, and set
\[
f = z^2 + x^2y + xy^n + xy^{n-r}z + 2G
\quad \text{or} \quad
f = z^2 + x^2y + y^n z + xy^{n-r}z + 2G,
\]
where $r,n$ are positive integers satisfying $1 \leq r<n$, and $G\in A$.
Let $\alpha\ge0$ be an integer.
Set $f_1:=f$ and
\[
f_i:=f\,\Delta_1(f)^{1+2+\cdots+2^{i-2}}\qquad(i\ge2).
\]
Define
\[
h:=\inf\{\,h'\ge1 \mid f_{h'}\notin(x^{2^{h'}},y^{2^{h'}(1+\alpha)},z^{2^{h'}},2)\,\}.
\]
Then the following hold:
\begin{enumerate}
  \item If the coefficient of $xz$ in $G$ is divisible by $2$, then
  \[
  h \leq e:=\inf\{e'\ge1 \mid 2^{e'}(1+\alpha)-2^{e'-1}-(n-r)\ge0\}.
  \]
  Furthermore, if 
  \[
f_h\in(xz)^{2^{h}-1}+(x^{2^{h}},y^{2^{h}(1+\alpha)},z^{2^{h}},2),
\]
then we have $h=e$
  and
  \[
  f_h\equiv ux^{2^h-1}y^{2^{h-1}-1+(n-r)}z^{2^h-1}
  \pmod{(x^{2^{h}},y^{2^{h}(1+\alpha)},z^{2^{h}},2)}
  \]
  for some unit $u \in A^\times$.
  \item If the coefficient of $xz$ in $G$ is not divisible by $2$, then
  \[
  h\leq e_{xz}:=\inf\{e'\ge1 \mid 2^{e'}(1+\alpha)-1-(n-r)\ge0\}.
  \]
  Furthermore, if 
  \[
f_h\in(xz)^{2^{h}-1}+(x^{2^{h}},y^{2^{h}(1+\alpha)},z^{2^{h}},2),
\]
then we have $h=e_{xz}$ and
  \[
  f_h\equiv ux^{2^h-1}y^{n-r}z^{2^h-1}
  \pmod{(x^{2^{h}},y^{2^{h}(1+\alpha)},z^{2^{h}},2)}
  \]
  for some unit $u \in A^\times$.
\end{enumerate}
\end{lemma}

\begin{proof}
We treat the case
\[
f=z^2 + x^2y + xy^n + xy^{n-r}z + 2G.
\]
The other case can be handled by the same argument.
We note that
\[
\Delta_1(f)\equiv \Delta_1(f-2G)+G^2 \pmod{2}
\]
and that $e_{xz} \leq e$.
We first prove the following claim.

\begin{claim}\label{claim:x-z}
Assume that at least one of the coefficients of $x$, $z$, or $1$ in $G$ is not divisible by $2$.
Then we have $h \leq e_{xz}$ and
\[
f_h\notin(xz)^{2^{h}-1}+(x^{2^{h}},y^{2^{h}(1+\alpha)},z^{2^{h}},2).
\]
\end{claim}

\begin{claimproof}
Denote by $G_x$, $G_z$, and $G_1$ the coefficients of $x$, $z$, and $1$ in $G$, respectively.
Assume that $G_x$ is not divisible by $2$.
The cases where $G_z$ or $G_1$ is not divisible by $2$ can be treated in the same way.

Since the coefficient of $x^2$ in $\Delta_1(f)$ is $G_x^2$, the coefficient of
$x^{2^{e_{xz}}-1}y^{n-r}z$ in $f_{e_{xz}}$ is
\[
G_x^{2+4+\cdots+2^{e_{xz}-1}}.
\]
By the definition of $e_{xz}$, we have
\[
n-r \leq 2^{e_{xz}}(1+\alpha)-1.
\]
Hence
\[
f_{e_{xz}} \notin (x^{2^{e_{xz}}},y^{2^{e_{xz}}(1+\alpha)},z^{2^{e_{xz}}},2),
\]
and therefore $h \leq e_{xz}$.

We next show that
\[
f_h\notin(xz)^{2^{h}-1}+(x^{2^{h}},y^{2^{h}(1+\alpha)},z^{2^{h}},2).
\]
Assume to the contrary that
\[
f_h\in(xz)^{2^{h}-1}+(x^{2^{h}},y^{2^{h}(1+\alpha)},z^{2^{h}},2).
\]
By the definition of $h$, we have
\[
f_h\equiv ux^{2^h-1}y^mz^{2^h-1}
\pmod{(x^{2^{h}},y^{2^{h}(1+\alpha)},z^{2^{h}},2)}
\]
for some $u \in A^\times$ and some integer $m \leq 2^h(1+\alpha)-1$.

On the other hand, since
\[
f_h=f\Delta_1(f)\Delta_1(f)^2\cdots \Delta_1(f)^{2^{h-2}},
\]
and the $x$- and $z$-orders of $x^{2^h-1}y^mz^{2^h-1}$ are odd, the monomial
$x^{2^h-1}y^mz^{2^h-1}$ must be divisible by $xy^{n-r}z$.
Therefore,
\[
m \geq n-r.
\]
In particular, we have
\[
n-r \leq m \leq 2^h(1+\alpha)-1.
\]

Moreover, by the argument above, the coefficient of
$x^{2^{h}-1}y^{n-r}z$ in $f_h$ is not divisible by $2$.
This contradicts the assumption that
\[
f_h\in(xz)^{2^{h}-1}+(x^{2^{h}},y^{2^{h}(1+\alpha)},z^{2^{h}},2).
\]
Thus,
\[
f_h\notin(xz)^{2^{h}-1}+(x^{2^{h}},y^{2^{h}(1+\alpha)},z^{2^{h}},2),
\]
as desired.
\end{claimproof}

\noindent
By \cref{claim:x-z}, we may assume that the coefficients of $x$, $z$, and $1$ in $G$ are all divisible by $2$.

Next, assume that the coefficient of $xz$ in $G$ is divisible by $2$, and we prove $(1)$.
If $h=1$, then the assertion is clear (note that $n \geq 2(1+\alpha)$ by the assumption), so we may assume $h\ge2$.
If $e=1$, then necessarily $h=1$, so we may also assume $e \geq 2$.
In particular, we have $n-r \geq 2$.
By the minimality of $e$, we have
\begin{equation}\label{eq:e1}
    n-r\ge2^{e-1}(1+\alpha)-2^{e-2} +1 \geq 2^{e-2}+1.
\end{equation}

Choose monomials $M_1,M_2,\ldots,M_{e-1}$ with coefficient $1$ such that the coefficient of $M_1$ in $f\Delta_1(f)$ is not divisible by $2$ and the coefficient of $M_i$ in $\Delta_1(f)^{2^{i-1}}$ is not divisible by $2$ for each $i\ge2$.

First, assume that
\[
M_1\cdots M_{e-1} =  x^{2^e-1}y^{2^{e-1}-1+(n-r)}z^{2^e-1}.
\]
Since the $x$- and $z$-orders of $M_1\cdots M_{e-1}$ are odd, $M_1$ must be divisible by $xy^{n-r}z$.
By comparing the $y$-orders, we obtain
\[
M_1=x^3y^{n-r+1}z^3
\]
and
\[
M_2 \cdots M_{e-1}=x^{2^e-4}y^{2^{e-1}-2}z^{2^e-4}.
\]
Again by comparing the $y$-orders, and using \eqref{eq:e1}, we see that
\[
M_2=(x^2yz^2)^2.
\]
Repeating the same argument, we obtain
\[
M_i=(x^2yz^2)^{2^{i-1}}
\]
for every $i \geq 2$.
We note that the coefficient of $(x^2yz^2)^{2^{i-1}}$ in $\Delta_1(f)^{2^{i-1}}$ is not divisible by $2$.
Hence the monomial
\[
x^{2^e-1}y^{2^{e-1}-1+(n-r)}z^{2^e-1}
\]
appears in $f_e$ with coefficient not divisible by $2$, and therefore $h\le e$ since
\[
2^{e-1}-1+(n-r) \leq 2^e(1+\alpha)-1
\]
by the definition of $e$.

Next, assume that
\[
f_h\in(xz)^{2^{h}-1}+(x^{2^{h}},y^{2^{h}(1+\alpha)},z^{2^{h}},2)
\]
and that
\[
M_1\cdots M_{h-1}=x^{2^h-1}y^{m}z^{2^h-1}
\]
for some $m\le 2^{h-1}-1+(n-r)$.
Since the $x$- and $z$-orders of $M_1\cdots M_{h-1}$ are odd, $M_1$ must be divisible by $xy^{n-r}z$.
Thus the $y$-order of $M_{h-1}$ is at most
\[
m-(n-r) \leq 2^{h-1}-1.
\]

First, assume that $h \geq 3$.
We note that
\[
\Delta_1(f)^{2^{h-2}}\equiv \Delta_1(f-2G)^{2^{h-2}}+G^{2^{h-1}} \pmod{2}.
\]
Since the coefficients of $x$, $z$, $1$, and $xz$ in $G$ are divisible by $2$, the monomial $M_{h-1}$ must appear in $\Delta_1(f-2G)^{2^{h-2}}$.
Furthermore, by \eqref{eq:e1}, we have
\[
2^{h-2}(n-r) \geq 2^{h-2}(2^{e-2}+1) \geq 2^{h-1}.
\]
Therefore, the $y$-order of $M_{h-1}$ is strictly smaller than $2^{h-2}(n-r)$, and hence
\[
M_{h-1}=(x^2yz^2)^{2^{h-2}}
\]
and
\[
M_1\cdots M_{h-2}=x^{2^{h-1}-1}y^{m-2^{h-2}}z^{2^{h-1}-1}.
\]
We note that
\[
m-2^{h-2} \leq 2^{h-2}-1+(n-r).
\]
Repeating the same argument, we obtain
\[
M_1=x^3y^{m-(2^{h-2}+2^{h-3}+\cdots+2)}z^3.
\]
Therefore,
\[
M_1=x^3y^{n-r + 1}z^3
\]
since the coefficients of $x$, $z$, $1$, and $xz$ in $G$ are divisible by $2$. Thus
\[
m=n-r+2^{h-2}+2^{h-3}+\cdots+2+ 1 =2^{h-1}-1+(n-r),
\]
and
\[
f_h\equiv ux^{2^h-1}y^{2^{h-1}-1+(n-r)}z^{2^h-1}
\pmod{(x^{2^{h}},y^{2^{h}(1+\alpha)},z^{2^{h}},2)}
\]
for some unit $u \in A^\times$.
By the definition of $h$, we have
\[
2^{h-1}-1+(n-r) \leq 2^h(1+\alpha)-1,
\]
and hence $e=h$.

Next, we prove $(2)$.
By the same argument as in the proof of $(1)$, we may assume $e_{xz}\ge2$.
Choose monomials $M_1,M_2,\ldots,M_{e_{xz}-1}$ as above, and assume that
\[
M_1\cdots M_{e_{xz}-1}=x^{2^{e_{xz}}-1}y^{n-r}z^{2^{e_{xz}}-1}.
\]
By the same reasoning as in the proof of $(1)$, $M_1$ must be divisible by $xy^{n-r}z$.
Since the coefficients of $x$, $z$, and $1$ in $G$ are divisible by $2$, we have
\[
M_1=x^3y^{n-r}z^3
\quad \text{and} \quad
M_i=(xz)^{2^i}
\]
for all $i\ge2$.
Hence the coefficient of $x^{2^{e_{xz}}-1}y^{n-r}z^{2^{e_{xz}}-1}$ in $f_{e_{xz}}$ is not divisible by $2$ and $h\le e_{xz}$.

Next, assume that
\[
f_h\in(xz)^{2^{h}-1}+(x^{2^{h}},y^{2^{h}(1+\alpha)},z^{2^{h}},2)
\]
and
\[
M_1\cdots M_{h-1}=x^{2^h-1}y^{m}z^{2^h-1}
\]
for some $m\le n-r$.
By the same argument as above, we have $m=n-r$, $M_1=x^3y^{n-r}z^3$, and $M_i=(xz)^{2^i}$ for $i\ge2$.
Thus $e=h$ and
\[
  f_h\equiv u\,x^{2^h-1}y^{n-r}z^{2^h-1}
  \pmod{(x^{2^{h}},y^{2^{h}(1+\alpha)},z^{2^{h}},2)}
\]
for some $u \in A^\times$, as desired.
\end{proof}

\subsection{Lifts of \texorpdfstring{$D_{2n}^r$}{D(2n)r} and \texorpdfstring{$D_{2n+1}^r$}{D(2n+1)r} without the monomial \texorpdfstring{$xz$}{xz}}

\begin{notation}\label{notation:D2nr-no-xz}
Let $p=2$, $A=W(k)[[x,y,z]]$, and $\m:=(x,y,z,2)$, and set
\[
f = z^2 + x^2y + xy^n + xy^{n-r}z + 2G
\quad \text{or} \quad
f = z^2 + x^2y + y^n z + xy^{n-r}z + 2G,
\]
where $r,n$ are integers satisfying $r<n$ and $G\in A$.
Assume moreover that the coefficient of $xz$ in $G$ is divisible by $2$.

We inductively define two sequences of integers $(\alpha_i)_{i\ge0}$ and $(e_i)_{i\ge0}$ as follows:
\begin{itemize}
  \item $\alpha_0=0$;
  \item for each $i\ge0$, having defined $\alpha_i$, let $e_i$ be the smallest  integer satisfying
  \[
    2^{e_i}(1+\alpha_i)-2^{e_i-1}-(n-r)\ge0,
  \]
  and then define
  \begin{equation}
  \label{eqn:alpha}
    \alpha_{i+1}=2^{e_i}(1+\alpha_i)-2^{e_i-1}-(n-r).
  \end{equation}
\end{itemize}
We note that 
\begin{equation}
\label{eqn:no-xz-e0}
e_0=\rup{\log_2(n-r)}+1.
\end{equation}
Let $(h_0,h_1,\ldots)$ denote the multi-height of $A/(f)$.
Set $f_1:=f$ and
\[
f_i:=f\Delta_1(f)^{1+2+\cdots+2^{i-2}}\qquad(i\ge2),
\]
as in Lemma \ref{f_h-lemma}.
For $i\ge0$, put
\[
E_i:=e_0+\cdots+e_{i-1}+e_i,\qquad
F_i:=f_{e_0}^{2^{e_1+\cdots+e_{i-1}+e_i}}\cdots f_{e_{i-1}}^{2^{e_i}}f_{e_i}.
\]
Moreover, we set $E_{-1} :=0$ and $F_{-1}:=1.$

For each $m\ge0$, we define the conditions $(\ast)_m$ and $(\dagger)_{m}$ by
\begin{align*}
    (\ast)_m:\quad& 
F_m \equiv u_m x^{2^{E_m}-1}y^{2^{E_m}-1-\alpha_{m+1}}z^{2^{E_m}-1}
\pmod{\m^{[2^{E_m}]}}
\quad\text{for some }u_m\in A^\times \\
(\dagger)_{m}:\quad& f_{h_m} \in (xz)^{2^{h_m}-1}+(x^{2^{h_m}},y^{2^{h_m}(1+\alpha_m)},z^{2^{h_m}},2).
\end{align*}
\end{notation}

\begin{remark}\label{rem:alpha-bound}
We have $0 \leq \alpha_i < n-r$ and $e_i \geq 1$ for all $i \geq 0$.
Indeed, we have $\alpha_0 = 0$ and $e_0 \geq 1$.
Assume that $\alpha_i \geq 0$.
Then, by the minimality of $e_i$, we obtain
\begin{equation}
\label{eqn:minimale}
n-r > 2^{e_i-1}(1+\alpha_i) - 2^{e_i-2}.
\end{equation}
It follows that
\[
0 \leq \alpha_{i+1}
= 2^{e_i}(1+\alpha_i) - 2^{e_i-1} - (n-r)
< n-r.
\]
Furthermore, since
\[
(1+\alpha_{i+1}) - \frac{1}{2} - (n-r) \leq -\frac{1}{2} < 0,
\]
we obtain $e_{i+1} \geq 1$.
Repeating this process, we conclude that $0 \leq \alpha_i < n-r$ and $e_i \geq 1$ for all $i \geq 0$.
\end{remark}

\begin{lemma}\label{F-E-lemma-no-xz}
We use \cref{notation:D2nr-no-xz}.
We fix an integer $m \geq -1$.
Moreover, if $m\geq 0$, we assume that $(\ast)_m$ holds and $h_i=e_i$ for all $0 \leq i \leq m$.
Then we have
\[
h_{m+1}=\inf\{\,h\mid f_h\notin(x^{2^h},y^{2^h(1+\alpha_{m+1})},z^{2^h},2)\,\}
\]
and $h_{m+1} \leq e_{m+1}$.
Furthermore, if $(\dagger)_{m+1}$ does not hold, then $h_{m+2}=h_{m+3}=\cdots=1$.
Otherwise, the condition $(\ast)_{m+1}$ holds and $h_{m+1}=e_{m+1}$.
\end{lemma}

\begin{proof}
By Proposition~\ref{non-ideal-ver}, we have
\[
h_{m+1}=\inf\{h\ge1 \mid F_m^{2^h}f_h\notin\m^{[2^{E_m+h}]}\}.
\]
For $h\ge1$, it follows that
\begin{equation}
\label{eqn:Fm+1}
F_m^{2^h}f_h \equiv
u_m^{2^{h}} x^{2^{E_m+h}-2^h}y^{2^{E_m+h}-2^h(1+\alpha_{m+1})}z^{2^{E_m+h}-2^h}f_h
\pmod{\m^{[2^{E_m+h}]}}.
\end{equation}
Thus,
\[
h_{m+1}=\inf\{\,h\ge1 \mid f_h\notin(x^{2^h},y^{2^h(1+\alpha_{m+1})},z^{2^h},2)\,\}.
\]
Set $h:=h_{m+1}$.
By \cref{f_h-lemma}(1), we have $h \leq e_{m+1}$.
If $(\dagger)_{m+1}$ does not hold, that is,
\[
f_h\notin(xz)^{2^h-1} +(x^{2^h},y^{2^h(1+\alpha_{m+1})},z^{2^h},2),
\]
then we have
\[
F^{2^h}_mf_{h} \cdot J \nsubseteq \m^{[2^{E_m+h}]},
\]
where $J=(x,z)$.
Since $x\in u(F_*(fxA))$ and $z\in u(F_*(fzA))$, by \cref{stable-1},  we have $h_{m+2}=h_{m+3}=\cdots=1$.
Next, we assume that $(\dagger)_{m+1}$ holds.
Since the coefficient of $xz$ in $G$ is divisible by $2$, it follows from \cref{f_h-lemma}~(1) that
\[
f_h \equiv u\,x^{2^h-1}y^{2^{h-1}-1+(n-r)}z^{2^h-1}
\pmod{(x^{2^h},y^{2^h(1+\alpha_{m+1})},z^{2^h} ,2)}
\]
for some unit $u\in A^\times$, and that $e_{m+1}=h$.
Substituting this into  (\ref{eqn:Fm+1}) and using (\ref{eqn:alpha}), we have
\[
F_{m+1} = F_m^{2^h} f_h  \equiv u_{m+1}x^{2^{E_{m+1}}-1}y^{2^{E_{m+1}}-1-\alpha_{m+2}}z^{2^{E_{m+1}}-1}
\pmod{\m^{[2^{E_{m+1}}]}}
\]
for some $u_{m+1} \in A^{\times}$,
that is, $(\ast)_{m+1}$ holds.
This completes the proof.
\end{proof}

\begin{theorem}\label{no-xz}
We use \cref{notation:D2nr-no-xz}.
Then the following hold:
\begin{enumerate}
  \item There exist integers $s<t$ such that $\alpha_{s+i}=\alpha_{t+i}$ for every $i\ge0$.
  \item $h_i\le e_i$ for every $i\ge0$.
  \item If $h_j<e_j$ for some $j\ge0$, then $h_i=1$ for all $i>j$.
  \item 
  Let $s<t$ be integers as in (1).
  If there exists $j\ge0$ such that $h_{t+j}\ge2$, then $e_i=h_i$ for every $i\ge0$.
\end{enumerate}
In particular, $(h_0, h_1, \ldots)$ is an element of a finite set
\[
([1,e_0]\times[1,e_1]\times\cdots\times[1,e_t]\times\{1\}\times\cdots)\cup\{(e_0,e_1,\ldots)\} \subseteq \prod \Z_{\geq 1}.
\]
Note that $\alpha_0, \alpha_1, \ldots$ and $e_0, e_1, \ldots, $ depend only on $n-r$.
\end{theorem}

\begin{proof}
By \cref{rem:alpha-bound}, there exist $s<t$ such that $\alpha_s=\alpha_t$.
Since $e_s$ and $\alpha_{s+1}$ (resp.\ $e_t$ and $\alpha_{t+1}$) are determined by $\alpha_s$ (resp.\ $\alpha_t$), we obtain $e_s=e_t$ and $\alpha_{s+1}=\alpha_{t+1}$.
Repeating this process, we get $\alpha_{s+i}=\alpha_{t+i}$ for all $i\ge0$, proving (1).

We first compute $h_0$.
By \cref{non-ideal-ver}, we have
\[
h_0=\inf\{h\ge1\mid f_h\notin\m^{[2^h]}\}.
\]
By the proof of \cref{F-E-lemma-no-xz}, it follows that $h_0\le e_0$.
Moreover, either $h_1=h_2=\cdots=1$, or $h_0=e_0$ and $(\ast)_0$ holds, namely
\[
f_{e_0}\equiv u_0x^{2^{e_0}-1}y^{2^{e_0}-1-\alpha_1}z^{2^{e_0}-1}\pmod{\m^{[2^{e_0}]}}
\]
for some unit $u_0\in A^\times$.

Next, we prove (2).
We fix an integer $m \geq -1$.
Moreover, if $m\geq 0$,
we assume that $(\ast)_i$ holds and $h_{i}=e_i$ for all $i\le m$.
Then by \cref{F-E-lemma-no-xz}, we have $h_{m+1}\le e_{m+1}$.
Furthermore, either $h_{m+2}=h_{m+3}=\cdots=1$, or   $(\ast)_{m+1}$ holds and $h_{m+1}=e_{m+1}$.
If $h_{m+2}=h_{m+3}=\cdots=1$, then (2) holds.
Thus, we may assume $(\ast)_{m+1}$ holds and $h_{m+1}=e_{m+1}$, thus repeating such a process, we obtain (2).
Furthermore, the assertion (3) follows from the above argument.

Finally, we prove (4).
Since $h_{t+j}\ge2$, it follows that  $(\ast)_{i}$, $(\dagger)_{i}$  and $e_i=h_i$ for $i\le t+j-1$ by \cref{F-E-lemma-no-xz}.
We prove that $(\ast)_m$, $(\dagger)_m$, and $e_m=h_m$ hold by induction on $m$.
By the induction hypothesis, \((\ast)_i\), \((\dagger)_i\), and \(e_i=h_i\) hold for all \(i \le m\).
Moreover, we may assume $m \geq t+j-1$.
Then we have
\[
h_{m+1}=\inf\{h\mid f_h\notin(x^{2^h},y^{2^h(1+\alpha_{m+1})},z^{2^h},2)\}
\]
by \cref{F-E-lemma-no-xz}.
Moreover, the right-hand side of the above equation coincides with $h_{m+1-(t-s)}$ since $m+1-(t-s)\le m$ and $\alpha_{m+1}=\alpha_{m+1-(t-s)}$ by the choice of $s$ and $t$.
Thus, we have $h_{m+1}=h_{m+1-(t-s)}$.
Since $(\dagger)_{m+1}$ depends only on $h_{m+1}$ and $\alpha_{m+1}$, the condition $(\dagger)_{m+1-(t-s)}$ implies the condition $(\dagger)_{m+1}$.
By \cref{F-E-lemma-no-xz}, we have $(\ast)_{m+1}$ and $e_{m+1}=h_{m+1}$ hold, as desired.
\end{proof}

In the special case where $G=0$, we can give below a formula to compute the multi-height (or $\ppt(-,p)$).

\begin{theorem}\label{D2nr}
We use \cref{notation:D2nr-no-xz} and assume $G=0$.
Then we have the following:
\begin{enumerate}
    \item If $\alpha_i < r$ for all $i \ge 0$, then $h_i = e_i$ for all $i \ge 0$, and 
    \[
    \ppt(A/(f),p) = \frac{1}{\,2(n - r) - 1\,}.
    \]
  \item If there exists $m \ge 0$ such that $\alpha_m \ge r$, then taking the minimal such $m$, we have
  \[
    h_i =
    \begin{cases}
      e_i & \text{for } 0 \le i \le m, \\
      1   & \text{for } i \ge m+1.
    \end{cases}
  \]
  \item The multi-height of $A/(f)$ coincides with the naive multi-height of the special fiber $\var{A/(f)} := A/(f) \otimes_{\Z} \Z/p\Z$.  
  Consequently, $\ppt(A/(f),p)$ equals the minimum of the perfectoid pure thresholds among the lifts of type $D^r_{2n}$ and $D^r_{2n+1}$.
\end{enumerate}
\end{theorem}

\begin{proof}
First, assume $\alpha_i < r$ for all $i \ge 0$.
By a computation as in the proof of \cref{f_h-lemma},
\[
f_{e_i} \equiv x^{2^{e_i}-1}y^{2^{e_i-1}-1+(n-r)}z^{2^{e_i}-1} \pmod{(x^{2^{e_i}},y^{2^{e_i}(1+\alpha_i)},z^{2^{e_i}},2)}
\]
since 
\[
2^{e_i-1}-1+n = 2^{e_i-1}-1+n-r+r \ge 2^{e_i-1}+n-r+\alpha_{i+1} = 2^{e_i}(1+\alpha_i)
\]
for every $i \ge 0$.
Therefore, by  \cref{F-E-lemma-no-xz}, we have $e_i=h_i$ for all $i \ge 0$
(note that $h_0 = e_0$ follows from \Cref{f_h-lemma} and \Cref{non-ideal-ver}).

Set 
\[
b_i:=\frac{1+2\alpha_i}{2(n-r)-1}
\]
for $i \ge 0$.
Then $0 < b_i \le 1$, since $0\leq \alpha_i \le n-r-1$.
By the definition of $\alpha_i$, we have
\[
2^{e_i}b_i = \frac{2^{e_i}(1+2\alpha_i)}{2(n-r)-1}
 = \frac{2((n-r)+\alpha_{i+1})}{2(n-r)-1}
 = 1 + b_{i+1}
\]
for every $i \ge 0$.
Considering the $p$-adic expansion of $b_i$, we obtain
\[
b_i = \frac{1}{2^{e_i}} + \frac{1}{2^{e_i+e_{i+1}}} + \cdots,
\]
and in particular $b_0 = \ppt(A/(f),p)$ by \cref{mult-ht-to-ppt}.
On the other hand,
\[
b_0 = \frac{1+2\alpha_0}{2(n-r)-1} = \frac{1}{2(n-r)-1}.
\]
Thus we obtain (1).

Next, we prove (3) under the assumption $\alpha_i<r$ for all $i \ge 0$.
Let $(\var{h}_0,\var{h}_1,\ldots)$ denote the naive multi-height of $\var{A/(f)}$.
We show that $\var{h_i}=e_i$ by induction on $i \ge 0$.

The following computation for the case $i=0$ has already been carried out in \cite{KTY25}*{Section~4.7}; however, we include it here for later use.
For $i=0$, if $e_0=1$, we have $e_0=\var{h}_0$
since $\var{A/(f)}$ is $F$-pure.
If $e_0 \ge 2$, then $fy^{\alpha_1} \in \var{I}_1$ and $u(F_*(fy^{\alpha_1}))=0$ because $n-r+\alpha_1$ is even.
Therefore,
\[
\theta(F_*(fy^{\alpha_1}))=xy^{2^{e_0-2}}z \in \var{I}_2.
\]
Since $xy^{2^{e_0-2}}z \in \Ker(u)$ if $e_0 \ge 3$, we have
\[
\theta^{e_0-1}(F^{e_0-1}_*(fy^{\alpha_1}))=xyz \in \var{I}_{e_0}.
\]
Hence $\var{h}_0 \le e_0$.
By (1) and \cref{multi-height-positive-char}, we obtain $\var{h}_0=e_0$, as desired.

Now assume $i \ge 1$.
By the same argument as above, we have $fy^{\alpha_{i+1}} \in \Ker(u)$ if $e_i \ge 2$, and
\[
\theta^{e_i-1}(F^{e_i-1}_*(fy^{\alpha_{i+1}}))=xy^{2\alpha_i+1}z \in \var{I}_{e_i}.
\]
Therefore,
\[
u(F_*xy^{2\alpha_i+1}z)f=fy^{\alpha_i} \in \var{I}_{(1,e_i)}.
\]
Repeating this process, we obtain $fy^{\alpha_j} \in \var{I}_{(1,e_j,e_{j+1},\ldots,e_i)}$ for $1 \leq j \leq i$ and $xyz \in \var{I}_{(e_0,e_1,\ldots,e_i)}$, which implies $\var{h}_i \le e_i$.
By (1) and \cref{multi-height-positive-char}, we obtain $\var{h}_i = e_i$ as desired.

Next, assume that there exists $m \ge 0$ such that $\alpha_m \ge r$, and take the minimal such $m$.
By the same argument as in the proof of (1), we have $h_i=e_i$ for $i \le m$.
Moreover, $(\ast)_i$ hold for $i\leq m-1$.
Furthermore, by a computation as in the proof of \cref{f_h-lemma},
\[
f_{e_m} \equiv x^{2^{e_m}-1}y^{2^{e_m-1}-1+(n-r)}z^{2^{e_m}-1} + g \pmod{(x^{2^{e_m}},y^{2^{e_m}(1+\alpha_m)},z^{2^{e_m}},2)},
\]
where
\[
g:=
\begin{cases}
    x^{2^{e_m}-1}y^{2^{e_m-1}-1+n}z^{2^{e_m}-2} & \text{if } f=z^2+x^2y+xy^n+xy^{n-r}z, \\
    x^{2^{e_m}-2}y^{2^{e_m-1}-1+n}z^{2^{e_m}-1} & \text{if } f=z^2+x^2y+y^nz+xy^{n-r}z.
\end{cases}
\]
In particular, since $(\dagger)_m$ does not hold, we have $h_i=1$ for $i \ge m+1$ by \cref{F-E-lemma-no-xz}, and thus (2) follows.

Finally, we prove (3) under the assumption of (2).
From the above argument, we have $\var{h}_i=e_i$ for $i \le m$.
Since $x \in u(F_*(xf\var{A}))$ and $z \in u(F_*(zf\var{A}))$, 
it follows that $x,z \in \var{I}_{(1,1,\ldots,1)}$.
Put
\[
g':=
\begin{cases}
y^{2^{e_m}\alpha_m+2^{e_m-1}-n}z & \text{if } f=z^2+x^2y+xy^n+xy^{n-r}z, \\
xy^{2^{e_m}\alpha_m+2^{e_m-1}-n} & \text{if } f=z^2+x^2y+y^nz+xy^{n-r}z.
\end{cases}
\]
Then we have $fg' \in \Ker(u)$, and
\[
\theta^i(F^i_*(fg'))=xy^{2^{e_m-i}\alpha_m+2^{e_m-1-i}}z \in \Ker(u)
\]
for $1 \le i \le e_m-2$.
Therefore,
\[
\theta^{e_m-1}(F^{e_m-1}_*(fg'))=xy^{2\alpha_m+1}z \in \var{I}_{(e_m,1,\ldots,1)}.
\]
Thus, we obtain $fy^{\alpha_m} \in \var{I}_{(1,e_m,1,\ldots)}$, and by (\ref{eqn:alpha}) and a similar computation as above, we have
\[
\theta^{e_{m-1}-1} (F_{*}^{e_{m-1}-1} (fy^{\alpha_m})) = xy^{2 \alpha_{m-1} +1}z \in  \var{I}_{(e_{m-1}, e_m, 1, \ldots, 1)}.
\]
Repeating this argument, we obtain 
\[
xy^{2 \alpha_0 +1}z = xyz \in  \var{I}_{(e_0, e_1, \ldots, e_m, 1,1, \ldots)}.
\]
Therefore, we have $\var{h}_i =1$ for any $i>m$ as desired.
\end{proof}

\subsection{Lifts of \texorpdfstring{$D_{2n}^r$}{D(2n)r} and \texorpdfstring{$D_{2n+1}^r$}{D(2n+1)r} with the monomial \texorpdfstring{$xz$}{xz}}

\begin{notation}\label{notation:D2nr-xz}
Let $p=2$, $A=W(k)[[x,y,z]]$, and $\m:=(x,y,z,2)$, and set
\[
f = z^2 + x^2y + xy^n + xy^{n-r}z + 2G
\quad \text{or} \quad
f = z^2 + x^2y + y^n z + xy^{n-r}z + 2G,
\]
where $r,n$ are integers satisfying $r<n$ and $G\in A$.
Assume moreover that the coefficient of $xz$ in $G$ is not divisible by $2$.

We inductively define two sequences of integers $(\alpha_i)_{i\ge0}$ and $(e_i)_{i\ge0}$ as follows:
\begin{itemize}
  \item $\alpha_0=0$;
  \item for each $i\ge0$, having defined $\alpha_i$, let $e_i$ be the smallest integer satisfying
  \[
    2^{e_i}(1+\alpha_i)-1-(n-r)\ge0,
  \]
  and then define
  \[
    \alpha_{i+1}=2^{e_i}(1+\alpha_i)-1-(n-r).
  \]
\end{itemize}
We note that 
\begin{equation}
\label{eqn:xz-e0}
e_0 = \rup{\log_2 (n-r+1)},
\end{equation}
 $e_i \geq 1$, and $n-r > \alpha_i \geq 0$ for all $i \geq 0$ by the same argument as in \cref{rem:alpha-bound}.
Let $(h_0,h_1,\ldots)$ denote the multi-height of $A/(f)$.
Set $f_1:=f$ and
\[
f_i:=f\,\Delta_1(f)^{1+2+\cdots+2^{i-2}}\qquad(i\ge2).
\]
For $i\ge0$, put
\[
E_i:=e_0+\cdots+e_{i-1}+e_i,\qquad
F_i:=f_{e_0}^{2^{e_1+\cdots+e_{i-1}+e_i}}\cdots f_{e_{i-1}}^{2^{e_i}}f_{e_i}.
\]
Moreover, we set $E_{-1} :=0$ and $F_{-1}:=1.$

For each $m\ge0$, we define the conditions $(\ast)_m$ and  $(\dagger)_{m}$
by
\begin{align*}
    (\ast)_m:\quad& 
F_m \equiv u_m x^{2^{E_m}-1}y^{2^{E_m}-1-\alpha_{m+1}}z^{2^{E_m}-1}
\pmod{\m^{[2^{E_m}]}}
\quad\text{for some }u_m\in A^\times \\
(\dagger)_{m}:\quad& f_{h_m} \in (xz)^{2^{h_m}-1}+(x^{2^{h_m}},y^{2^{h_m}(1+\alpha_m)},z^{2^{h_m}},2).
\end{align*}
\end{notation}

\begin{lemma}\label{F-E-lemma-xz}
We use \cref{notation:D2nr-xz}.
We fix an integer $m \geq -1$.
Moreover, if $m\geq 0$, we
assume that $(\ast)_m$ holds and $h_i=e_i$ for all $ 0 \leq i \leq m$.
Then
\[
h_{m+1}=\inf\{h\mid f_h\notin(x^{2^h},y^{2^h(1+\alpha_{m+1})},z^{2^h},2)\}
\]
and $h_{m+1} \leq e_{m+1}$.
Furthermore, if $(\dagger)_{m+1}$ does not hold, then $h_{m+2}=h_{m+3}=\cdots=1$.
Otherwise,  $h_{m+1}=e_{m+1}$ and $(\ast)_{m+1}$ hold.
\end{lemma}

\begin{proof}
The argument is identical to that of \cref{F-E-lemma-no-xz}, 
except for the use of \cref{f_h-lemma}~(2) instead of \cref{f_h-lemma}~(1), 
which reflects the presence of the monomial $xz$ in $G$.

Indeed, under the same setup as in \cref{F-E-lemma-no-xz}, 
we may assume
\[
f_h\in(xz)^{2^h-1}+(x^{2^h},y^{2^h(1+\alpha_{m+1})},z^{2^h},2).
\]
Then, by \cref{f_h-lemma}~(2), we obtain
\[
f_h \equiv u\,x^{2^h-1}y^{n-r}z^{2^h-1}
\pmod{(x^{2^h},y^{2^h(1+\alpha_{m+1})},z^{2^h},2)}
\]
for some unit $u\in A^\times$, and $e_{m+1}=h$.
Substituting this into the definition of $F_{m+1}$ yields
\[
F_{m+1} \equiv u_{m+1}x^{2^{E_{m+1}}-1}y^{2^{E_{m+1}}-1-\alpha_{m+2}}z^{2^{E_{m+1}}-1}
\pmod{\m^{[2^{E_{m+1}}]}}
\]
for some $u_{m+1} \in A^{\times}$,
that is, $(\ast)_{m+1}$ holds.
This completes the proof.
\end{proof}

\begin{theorem}\label{xz}
We use \cref{notation:D2nr-xz}.
Then the following hold:
\begin{enumerate}
  \item There exist integers $s<t$ such that $\alpha_{s+i}=\alpha_{t+i}$ for every $i\ge0$.
  \item $h_i\le e_i$ for every $i\ge0$.
  \item If $h_j<e_j$ for some $j\ge0$, then $h_i=1$ for all $i>j$.
  \item 
  Let $s<t$ be integers as in (1).
  If there exists $j\ge0$ such that $h_{t+j}\ge2$, then $e_i=h_i$ for every $i\ge0$.
\end{enumerate}
In particular, $(h_0, h_1, \ldots)$ is an element of a finite set
\[
([1,e_0]\times[1,e_1]\times\cdots\times[1,e_t]\times\{1\}\times\cdots)\cup\{(e_0,e_1,\ldots)\} \subseteq \prod \Z_{\geq 1}.
\]
Note that $\alpha_0, \alpha_1, \ldots$ and $e_0, e_1, \ldots, $ depend only on $n-r$.
\end{theorem}

\begin{proof}
The proof is identical to that of \cref{no-xz}, 
replacing \cref{F-E-lemma-no-xz}  with 
\cref{F-E-lemma-xz}.

In particular, by the same argument as in the proof of \cref{no-xz}, 
we obtain the periodicity of $(\alpha_i)$, that is, there exist integers $s<t$ such that 
$\alpha_{s+i}=\alpha_{t+i}$ for all $i\ge0$.
Furthermore, if $h_{t+j}\ge2$ for some $j\ge0$, 
then by the same inductive argument, 
we conclude that $h_i=e_i$ for all $i\ge0$.
Hence we obtain the desired inclusion.
\end{proof}

\begin{theorem}\label{D2nr-xz}
We use \cref{notation:D2nr-xz} and assume $G=xz$.
If \(\alpha_i < r\) for all \(i \ge 0\), then \(h_i = e_i\) for all \(i \ge 0\) and 
\[
\ppt(A/(f),p) = \frac{1}{n - r}.
\]
\end{theorem}

\begin{proof}
Since $\alpha_i < r$ for all $i \ge 0$, by a computation as in the proof of \cref{f_h-lemma}, we have
\[
f_{e_i} \equiv 
u_i
x^{2^{e_i}-1}y^{n-r}z^{2^{e_i}-1} \pmod{(x^{2^{e_i}},y^{2^{e_i}(1+\alpha_i)},z^{2^{e_i}},2)}
\]
for some unit $u_i \in A^\times$ by
\[
n = n-r+r \ge n-r+\alpha_{i+1}+1 = 2^{e_i}(1+\alpha_i)
\]
for every $i \ge 0$.
Therefore, by \cref{F-E-lemma-xz}, we have $e_i=h_i$ for all $i \ge 0$.

Set 
\[
b_i:=\frac{1+\alpha_i}{n-r}
\]
for $i \ge 0$.
Then $0 < b_i \le 1$, since $\alpha_i \le n-r-1$.
By the definition of $\alpha_i$, we have
\[
2^{e_i}b_i = \frac{2^{e_i}(1+\alpha_i)}{n-r}
 = \frac{n-r+1+\alpha_{i+1}}{n-r}
 = 1 + b_{i+1}
\]
for every $i \ge 0$.
Considering the $p$-adic expansion of $b_i$, we obtain
\[
b_i = \frac{1}{2^{e_i}} + \frac{1}{2^{e_i+e_{i+1}}} + \cdots,
\]
and in particular $b_0 = \ppt(A/(f),p)$ by \cref{mult-ht-to-ppt}.
On the other hand,
\[
b_0 = \frac{1+\alpha_0}{n-r} = \frac{1}{n-r},
\]
as desired.
\end{proof}

\section{Rationality and the Ascending Chain Condition}

\begin{theorem}\label{ACC}
Let $p$ be a prime number and $k$ an algebraically closed field of characteristic $p$.  
Consider the set
\[
\Sigma := \{ \ppt(R,p) \mid \text{$R$ is a $W(k)$-lift of an RDP of characteristic $p$ over $k$} \}.
\]
Then $\Sigma$ is contained in $\Q$ and satisfies the ascending chain condition.  
Moreover, if $p=2$, the set $\Sigma$ contains all numbers of the form $1/m$ with $m \in \Z_{\ge 1}$.    
In addition, every accumulation point of $\Sigma$ is $0$.
\end{theorem}

\begin{proof}
If $p \ne 2$, the result follows from Table~\ref{table:ppts}.
Hence we may assume $p=2$.
The inclusion $\Sigma \subseteq \Q$ follows from \cref{no-xz,xz}.
Furthermore, the inclusion
\[
\{1/m \mid m \in \Z_{\ge 1}\} \subseteq \Sigma 
\]
follows from \cref{D2nr-xz}.
Indeed, if we take $n=2m$, $r=m$, then we have $\alpha_i< n-r=r$ for all $i \geq 0$ and $\ppt(R,p)=1/m$.

Next, we prove that $\Sigma$ satisfies the ascending chain condition.
Suppose, for contradiction, that there exists a strictly increasing sequence
$P_1 < P_2 < P_3 < \cdots$ in $\Sigma$.
Then there exist elements $f_j \in W(k)[[x,y,z]] =: A$ such that
$\ppt(A/(f_j),p) = P_j$ for each $j$, and the reduction $\var{f}_j$
is one of the equations listed in Table~\ref{table:ppts} or of type $D^r_{2n}$ and $D_{2n+1}^r$.
By Table~\ref{table:ppts}, the set
\[
\{\, P_j \mid \text{$\var{A/(f_j)}$ is neither of type $D^r_{2n}$ nor of type $D^r_{2n+1}$ for any $n>r>0$} \,\}
\]
satisfies the ascending chain condition.
Thus, we may assume that for all $j \ge 1$, there exist integers $n_j > r_j > 0$
such that $\var{A/(f_j)}$ is of type $D^{r_j}_{2n_j}$ or $D^{r_j}_{2n_j+1}$.

Since the set $\{1/m \mid m \in \Z_{\ge 1}\}$ satisfies the ascending chain condition,
we may assume $P_j \notin \{1/m \mid m \in \Z_{\ge 1}\}$ for all $j \ge 1$.
In particular, the multi-height $\boldsymbol{h}^{(j)}$ of $A/(f_j)$ eventually becomes~$1$
for each $j \ge 1$ by \cref{no-xz,xz} and the proofs of \cref{D2nr,D2nr-xz}.
Set
\[
s_j := \inf\{\, s \mid h^{(j)}_i = 1 \text{ for all } i \ge s \,\}.
\]
If $s_j = 1$, then $P_j = 1/2^{h^{(j)}_0 - 1}$, and hence
\[
\{ P_j \mid s_j = 1 \} \subseteq \{1/m \mid m \in \Z_{\ge 1}\}.
\]
Therefore, we may assume $s_j \ge 2$ for all $j \ge 1$.

Next, we assume $f_j$ satisfies the conditions in \cref{notation:D2nr-no-xz} for every $j \geq 1$.
If 
\[
h^{(j)}_0=\rup{\log_2(n_j - r_j)} + 1,
\]
for all $j \geq 1$, then we have $n_1-r_1 \geq n_2-r_2 \geq \cdots $ since $P_1  < P_2 < \cdots$.
In particular, we have 
\[
\rup{\log_2(n_j - r_j)} + 1 \le \rup{\log_2(n_1 - r_1)} + 1
\]
for every $j \geq 1$.
Consequently, the set $\{ n_j - r_j \mid j \ge 1 \}$ is finite.
By \cref{no-xz}, the set $\{ P_j \mid j \ge 1 \}$ is also finite, which gives a contradiction.
Thus, we may assume $h^{(j)}_0< \rup{\log_2(n_j - r_j)} + 1$ for every $j \geq 1$.
Then we have $s_j=2$ by \cref{no-xz}, thus the set $\{ P_j \mid j \ge 1 \}$ is also finite.

On the other hand, if $f_j$ satisfies the conditions in \cref{notation:D2nr-xz} for every $j \geq 1$, then the finiteness of $\{ P_j \mid j \ge 1 \}$ follows from the same argument as above using \cref{xz} instead of \cref{no-xz}.

Finally, we prove the final assertion.
Suppose that there exists a non-zero accumulation point $P$ of $\Sigma$.
We take a sequence $\{P_j\}_{j \geq 1}$ in $\Sigma$ such that $\lim P_j=P$  and $P_j \neq P$ for any $j$.
We take a lift $R_j$ of a RDP such that $\ppt(R_j,p)=P_j$ for every $j \geq 1$.
Let $\boldsymbol{h}(j):=(h^{(j)}_0,h^{(j)}_1,\ldots)$ be the multi-height of $R_j$.
We take the $2$-adic expansion $P=\sum_{m=1} \frac{c_m}{2^m}$ such that the sequence $\{c_m\}_{m \geq 1}$ is not eventually zero.
Let $h \geq 1$ be the minimal integer such that $c_h \neq 0$.
Thus, we may assume $h^{(j)}_0=h$ for every $j \geq 1$.
Since the quasi-$F$-split height of $R$ determines $\ppt(R,p)$ if $R$ is of type $D_m^0$ for every $m \geq 1$, we may assume $R_j$ is of type $D_{2n_{j}}^{r_j}$ or $D_{2n_{j}+1}^{r_j}$.
By \cref{no-xz,xz} (cf.\ (\ref{eqn:no-xz-e0}, \ref{eqn:xz-e0})), if 
\[
h \neq \rup{\log_2 (n_j-r_j+1)}
\quad\text{and}\quad
h \neq \rup{\log_2(n_j-r_j)} + 1,
\]
then we 
then we have $\boldsymbol{h}{(j)}=(h,1,1,\ldots)$.
Therefore, we may assume 
\[
h=\rup{\log_2 (n_j-r_j+1)}
\quad \text{or} \quad \rup{\log_2(n_j - r_j)} + 1
\]
for every $j \geq 1$.
In particular, $n_j-r_j$ is bounded, and in particular, the set $\{\ppt(R_j,p)\}$ is finite by \cref{no-xz,xz}.
Therefore, we obtain the contradiction.
\end{proof}

\begin{example}
Since the list of possible values of $\ppt(-,p)$ for RDPs of type~$D$ has not been obtained, $\Sigma$ has not been determined completely when $p=2$.
For reference, we record the list of $t_r = t_{n,r} := \ppt (-,p)$ of the natural lift (i.e.,\ the lift where $G=0$) of type $D_{2n}^r$, for $n=17$.
This can be computed using a computer algebra system (combining with \cref{prop:repeating_multiheight} (cf.\ \cite{Takamatsu_code}) or \cref{D2nr}).
    \begin{align*}
    &(t_1, t_2, t_3, t_4, t_5, t_6, t_7, t_8, t_9, t_{10}, t_{11}, t_{12}, t_{13}, t_{14}, t_{15}) \\
    =& \bigl(\frac{1}{31}, \frac{9}{256}, \frac{5}{128}, \frac{41}{1024}, \frac{3}{64}, \frac{1}{21}, \frac{27}{512}, \frac{1}{17}, \frac{1}{15}, \frac{1}{13}, \frac{1}{11}, \frac{1}{9}, \frac{1}{7}, \frac{1}{5}, \frac{1}{3}\bigr).
    \end{align*}
Note that, in general, we have the following:
\begin{enumerate}
    \item 
When $\lceil \log_{2} (n-r) \rceil = \lceil \log_2 n \rceil$, we have $t_{n,r} = 1/2^{\lceil \log_2 (n-r) \rceil}$.
    \item 
When $r\geq \lfloor \frac{n}{2} \rfloor$, we have
$t_{n,r} = \frac{1}{2(n-r)-1}$.
   \item 
When $n-r$ is a power of $2$, we have 
$t_{n,r} = \frac{1}{2 (n-r) -1}$.
\end{enumerate}
Indeed, (1) follows from \cref{D2nr}(2) since we have $\alpha_1 \geq r$ in this case.
Also, (3) follows from \cref{D2nr}(1) since we have $\alpha_i = 0$ for any $i \geq 0$ in this case.
Moreover, (2) follows from \cref{D2nr}(1). Indeed, if $r \geq n-r$, we have $\alpha_i < n-r \leq r$ by \cref{rem:alpha-bound}.
We assume that $r< n-r$, i.e.,\ $n=2r+1$ in the following.
We shall show $\alpha_i < r$ by induction on $i$. Assume that $\alpha_i < r.$
If $e_i =1$, then we have
\[
\alpha_{i+1} = 2(1+ \alpha_i) - 1- (r+1) = 2\alpha_i -r <r.
\]
On the other hand, if $e_i \geq 2$, by (\ref{eqn:minimale}), we have
\[
n-r -1 \geq 2^{e_i-1}(1+\alpha_i) - 2^{e_i-2}.
\]
Therefore, we have
\[
\alpha_{i+1} = 2^{e_i} (1+ \alpha_i) -2^{e_i-1} - (n-r) \leq 2(n-r-1) - (n-r) = n-r-2 = r-1
\]
as desired.

\end{example}

\bibliographystyle{skalpha}
\bibliography{bibliography.bib}
\end{document}